\documentclass{amsart}
\usepackage{amssymb,amsmath,stmaryrd,mathrsfs,amsthm}
\usepackage[all,2cell]{xy}
\usepackage[neveradjust]{paralist}
\usepackage{enumitem}
\usepackage{color}
\definecolor{darkgreen}{rgb}{0,0.45,0}
\usepackage[pagebackref,colorlinks,citecolor=darkgreen,linkcolor=darkgreen]{hyperref}
\usepackage[noadjust]{cite}
\usepackage{subfigure}
\usepackage{array}
\usepackage{mathtools}
\usepackage{tikz}
\usepackage{xspace}
\usepackage{needspace}

\newcommand{\calO}{\ensuremath{\mathcal{O}}}
\newcommand{\calP}{\ensuremath{\mathcal{P}}}

\newcommand{\bbZ}{\ensuremath{\mathbb{Z}}}

\newcommand{\bC}{\ensuremath{\mathbf{C}}}
\newcommand{\bD}{\ensuremath{\mathbf{D}}}

\newcommand{\bS}{\ensuremath{\mathbf{S}}}

\newcommand{\fc}{\ensuremath{\mathfrak{c}}}

\newcommand{\fs}{\ensuremath{\mathfrak{s}}}

\newcommand{\ten}{\ensuremath{\otimes}}
\newcommand{\inv}{^{-1}}

\newcommand{\id}{\ensuremath{\operatorname{id}}}
\newcommand{\Hom}{\ensuremath{\operatorname{Hom}}}
\newcommand{\Ho}{\ensuremath{\operatorname{Ho}}}
\newcommand{\colim}{\ensuremath{\operatorname{colim}}}

\newdir{ >}{{}*!/-10pt/@{>}}
\newcommand{\iso}{\cong}

\newcommand{\too}[1][]{\ensuremath{\overset{#1}{\longrightarrow}}}

\newcommand{\toto}{\ensuremath{\rightrightarrows}}

\newcommand{\maps}{\colon}

\let\xto\xrightarrow

\def\xiso#1{\mathrel{\mathrlap{\smash{\xto[\smash{\raisebox{1.3mm}{$\scriptstyle\sim$}}]{#1}}}\hphantom{\xto{#1}}}}
\makeatletter
\def\slashedarrowfill@#1#2#3#4#5{%
  $\m@th\thickmuskip0mu\medmuskip\thickmuskip\thinmuskip\thickmuskip
   \relax#5#1\mkern-7mu%
   \cleaders\hbox{$#5\mkern-2mu#2\mkern-2mu$}\hfill
   \mathclap{#3}\mathclap{#2}%
   \cleaders\hbox{$#5\mkern-2mu#2\mkern-2mu$}\hfill
   \mkern-7mu#4$%
}
\def\rightslashedarrowfill@{%
  \slashedarrowfill@\relbar\relbar\mapstochar\rightarrow}
\newcommand\xslashedrightarrow[2][]{%
  \ext@arrow 0055{\rightslashedarrowfill@}{#1}{#2}}
\makeatother
\def\hto{\xslashedrightarrow{}}

\newtheorem{thm}{Theorem}[section]
  
\newtheorem{cor}{Corollary}
  \makeatletter\let\c@cor\c@thm\makeatother
  \numberwithin{cor}{section}
  
\newtheorem{prop}{Proposition}
  \makeatletter\let\c@prop\c@thm\makeatother
  \numberwithin{prop}{section}
  
\newtheorem{lem}{Lemma}
  \makeatletter\let\c@lem\c@thm\makeatother
  \numberwithin{lem}{section}
  
\theoremstyle{definition}
\newtheorem{defn}{Definition}
  \makeatletter\let\c@defn\c@thm\makeatother
  \numberwithin{defn}{section}

  \makeatletter\let\c@notn\c@thm\makeatother
  \numberwithin{notn}{section}
  
\theoremstyle{remark}
\newtheorem{rmk}{Remark}
  \makeatletter\let\c@rmk\c@thm\makeatother
  \numberwithin{rmk}{section}
  
\newtheorem{eg}{Example}
  \makeatletter\let\c@eg\c@thm\makeatother
  \numberwithin{eg}{section}

  \makeatletter\let\c@egs\c@thm\makeatother
  \numberwithin{egs}{section}

\makeatletter
\let\c@equation\c@thm
\makeatother
\numberwithin{equation}{section}

\makeatletter
\def\alwaysmath#1{\expandafter\expandafter\expandafter\global\expandafter\expandafter\expandafter\let\expandafter\expandafter\csname your@#1\endcsname\csname #1\endcsname
  \expandafter\def\csname #1\endcsname{\ensuremath{\csname your@#1\endcsname}}}
\alwaysmath{alpha}
\alwaysmath{beta}
\alwaysmath{gamma}
\alwaysmath{Gamma}
\alwaysmath{delta}
\alwaysmath{Delta}
\alwaysmath{epsilon}
\newcommand{\ep}{\ensuremath{\varepsilon}}
\alwaysmath{zeta}
\alwaysmath{eta}
\alwaysmath{theta}
\alwaysmath{Theta}
\alwaysmath{iota}
\alwaysmath{kappa}
\alwaysmath{lambda}
\alwaysmath{Lambda}
\alwaysmath{mu}
\alwaysmath{nu}
\alwaysmath{xi}
\alwaysmath{pi}
\alwaysmath{rho}
\alwaysmath{sigma}
\alwaysmath{Sigma}
\alwaysmath{tau}
\alwaysmath{upsilon}
\alwaysmath{Upsilon}
\alwaysmath{phi}
\alwaysmath{Phi}

\alwaysmath{chi}
\alwaysmath{psi}
\alwaysmath{Psi}
\alwaysmath{omega}
\alwaysmath{Omega}
\alwaysmath{ell}
\alwaysmath{infty}
\alwaysmath{sslash}
\alwaysmath{odot}
\makeatother

\let\al\alpha

\makeatletter
\DeclareRobustCommand\widecheck[1]{{\mathpalette\@widecheck{#1}}}
\def\@widecheck#1#2{%
    \setbox\z@\hbox{\m@th$#1#2$}%
    \setbox\tw@\hbox{\m@th$#1%
       \widehat{%
          \vrule\@width\z@\@height\ht\z@
          \vrule\@height\z@\@width\wd\z@}$}%
    \dp\tw@-\ht\z@
    \@tempdima\ht\z@ \advance\@tempdima2\ht\tw@ \divide\@tempdima\thr@@
    \setbox\tw@\hbox{%
       \raise\@tempdima\hbox{\scalebox{1}[-1]{\lower\@tempdima\box
\tw@}}}%
    {\ooalign{\box\tw@ \cr \box\z@}}}
\makeatother

\newcommand{\rdual}[1]{{{#1}^{\bigstar}}}

\newcommand{\tr}{\ensuremath{\operatorname{tr}}}

\newcommand{\fii}{\ensuremath{\mathfrak{i}}}

\newcommand{\Sip}{\ensuremath{\Sigma^\infty_+}}

\newcommand{\Set}{\ensuremath{\mathbf{Set}}}
\newcommand{\Ab}{\ensuremath{\mathbf{Ab}}}
\newcommand{\Sup}{\ensuremath{\mathbf{Sup}}}

\newcommand{\bRel}{\ensuremath{\mathbf{Rel}}}
\newcommand{\Top}{\ensuremath{\mathbf{Top}}}
\newcommand{\Sp}{\ensuremath{\mathbf{Sp}}}
\newcommand{\Vect}[1]{\ensuremath{\mathbf{Vect}_{#1}}}
\newcommand{\bMod}[1]{\ensuremath{\mathbf{Mod}_{#1}}}
\newcommand{\GrVect}[1]{\ensuremath{\mathbf{GrVect}_{#1}}}
\newcommand{\bGrMod}[1]{\ensuremath{\mathbf{GrMod}_{#1}}}
\newcommand{\bCh}[1]{\ensuremath{\mathbf{Ch}_{#1}}}
\newcommand{\gTop}[1]{\ensuremath{{#1}\text{-}\mathbf{Top}}}
\newcommand{\gSp}[1]{\ensuremath{{#1}\text{-}\mathbf{Sp}}}
\newcommand{\bEx}[1]{\ensuremath{\mathbf{Sp}_{#1}}}
\newcommand{\bCob}[1]{\ensuremath{{#1}\mathbf{Cob}}}

\def\calMat(#1){\ensuremath{\mathcal{M}\mathit{at}(\mathbf{#1})/\!_{\mathcal{S}\mathit{et}}}}

\makeatletter
\def\calSpan{\futurelet\next@char\@calSpan}
\def\@calSpan{\if\next@char(\def\span@next{\@@calSpan}\else\def\span@next{\@@calSpan(S)}\fi\span@next}
\def\@@calSpan(#1){\ensuremath{\mathbf{#1}/\!_{\mathbf{#1}}}}

\makeatletter
\def\Span{\futurelet\next@char\@bbSpan}
\def\@bbSpan{\if\next@char(\def\span@next{\@@bbSpan}\else\def\span@next{\@@bbSpan(S)}\fi\span@next}
\def\@@bbSpan(#1){\ensuremath{\mathbf{#1}\mathord{\sslash}\!_{\mathbf{#1}}}}

\def\Mat(#1){\ensuremath{\mathbb{M}\mathbf{at}(#1)\mathord{\sslash}\!_{\mathbb{S}\mathbf{et}}}}

\UseAllTwocells
\usetikzlibrary{scopes,calc,arrows,patterns}
\tikzset{ed/.style={auto,inner sep=0pt,font=\scriptsize}} 
\tikzset{>=stealth'}
\tikzset{vert/.style={draw,circle,inner sep=1pt,fill=white}}
\tikzset{vert2/.style={draw,circle,inner sep=2pt,fill=white}}

\colorlet{myblue}{blue!40!white}
\colorlet{myred}{red!35!white}
\colorlet{mygreen}{green!30!white}
\colorlet{myyellow}{yellow!10!white}
\tikzset{bluefill/.style={fill=myblue}}
\tikzset{redfill/.style={fill=myred}}
\tikzset{greenfill/.style={fill=mygreen}}
\tikzset{yellowfill/.style={fill=myyellow}}
\tikzset{dotsF/.style={pattern=north east lines,pattern color=black!60!white}}
\tikzset{dotsG/.style={pattern=north west lines,pattern color=black!60!white}}

\colorlet{mydarkred}{red!80!black}
\colorlet{mydarkblue}{blue!70!black}

\usetikzlibrary{decorations.pathmorphing}
\usetikzlibrary{decorations}
\tikzset{transf/.style={decorate,decoration={zigzag,amplitude=1pt,segment length=3pt}}}

\pgfdeclarelayer{background}
\pgfdeclarelayer{foreground}
\pgfsetlayers{background,main,foreground}

\newenvironment{tikzcenter}{\begin{center}\begin{tikzpicture}}{\end{tikzpicture}\end{center}}

\usetikzlibrary{shapes.geometric}
\tikzset{fiber/.style={draw,rectangle,inner sep=2pt,font=\scriptsize}}
\tikzset{pb/.style={draw,regular polygon,regular polygon sides=3,inner sep=0pt,shape border rotate=180,font=\scriptsize}}
\tikzset{pbe/.style={draw,regular polygon,regular polygon sides=3,inner sep=2pt,shape border rotate=180}} 
\tikzset{pbsm/.style={draw,regular polygon,regular polygon sides=3,inner sep=-.5pt,shape border rotate=180,font=\scriptsize}} 
\tikzset{pbflat/.style={draw,regular polygon,regular polygon sides=3,shape border rotate=270,inner sep=1pt,font=\scriptsize}}
\tikzset{pbeflat/.style={draw,regular polygon,regular polygon sides=3,shape border rotate=270,inner sep=2pt}}
\tikzset{pbsmflat/.style={draw,regular polygon,regular polygon sides=3,shape border rotate=270,inner sep=0pt,font=\scriptsize}}
\tikzset{pf/.style={draw,regular polygon,regular polygon sides=3,inner sep=0pt,font=\scriptsize}}
\tikzset{pfe/.style={draw,regular polygon,regular polygon sides=3,inner sep=2pt}} 
\tikzset{pfsm/.style={draw,regular polygon,regular polygon sides=3,inner sep=-.5pt,font=\scriptsize}} 
\tikzset{pfflat/.style={draw,regular polygon,regular polygon sides=3,shape border rotate=90,inner sep=1pt,font=\scriptsize}}
\tikzset{pfeflat/.style={draw,regular polygon,regular polygon sides=3,shape border rotate=90,inner sep=2pt}}
\tikzset{pfsmflat/.style={draw,regular polygon,regular polygon sides=3,shape border rotate=90,inner sep=0pt,font=\scriptsize}}

\tikzset{backwards/.style={z={(-.3,-.7)},x={(-1,0)},y={(0,-1.5)}}}
\tikzset{strings/.style={scale=.75, z={(0,-.4)},x={(-1,0)},y={(0,-2)}}}

\makeatletter

\newif\iftikz@to@relp
\tikz@to@relpfalse
\newif\iftikz@to@relpp
\tikz@to@relppfalse
\tikzoption{abs}[]{\tikz@to@relpfalse\tikz@to@relppfalse}
\tikzoption{rel}[]{\tikz@to@relptrue\tikz@to@relppfalse}
\tikzoption{rrel}[]{\tikz@to@relpfalse\tikz@to@relpptrue}

\tikzstyle{every curve to}=          []
\tikzstyle{curve to}=                [to path=\tikz@to@curve@path]

\tikzoption{bend angle}{\def\tikz@to@bend{#1}}

\tikzoption{bend left}[]{%
  \def\pgf@temp{#1}%
  \ifx\pgf@temp\pgfutil@empty%
  \else%
    \def\tikz@to@bend{#1}%
  \fi%
  \let\tikz@to@out=\tikz@to@bend%
  \c@pgf@counta=180\relax%
  \advance\c@pgf@counta by-\tikz@to@out\relax%
  \edef\tikz@to@in{\the\c@pgf@counta}%
  \tikz@to@switch@on%
  \tikz@to@relativetrue%
}

\tikzoption{bend right}[]{%
  \def\pgf@temp{#1}%
  \ifx\pgf@temp\pgfutil@empty%
  \else%
    \def\tikz@to@bend{#1}%
  \fi%
  \c@pgf@counta=\tikz@to@bend\relax%
  \c@pgf@counta=-\c@pgf@counta\relax%
  \edef\tikz@to@out{\the\c@pgf@counta}%
  \c@pgf@counta=180\relax%
  \advance\c@pgf@counta by-\tikz@to@out\relax%
  \edef\tikz@to@in{\the\c@pgf@counta}%
  \tikz@to@switch@on%
  \tikz@to@relativetrue%
}

\tikzoption{relative}[true]{\csname tikz@to@relative#1\endcsname}
\newif\iftikz@to@relative
\tikz@to@relativefalse

\tikzoption{in}{\def\tikz@to@in{#1}\tikz@to@switch@on}
\tikzoption{out}{\def\tikz@to@out{#1}\tikz@to@switch@on}

\tikzoption{in looseness}{\tikz@to@set@in@looseness{#1}}
\tikzoption{out looseness}{\tikz@to@set@out@looseness{#1}}
\tikzoption{looseness}{\tikz@to@set@in@looseness{#1}\tikz@to@set@out@looseness{#1}}

\tikzoption{in control}{\tikz@to@set@in@control{#1}}
\tikzoption{out control}{\tikz@to@set@out@control{#1}}
\tikzoption{controls}{\tikz@to@parse@controls#1\pgf@stop}

\tikzoption{in min distance}{\tikz@to@set@distances{#1}{}{}{}}
\tikzoption{in max distance}{\tikz@to@set@distances{}{#1}{}{}}
\tikzoption{in distance}{\tikz@to@set@distances{#1}{#1}{}{}}
\tikzoption{out min distance}{\tikz@to@set@distances{}{}{#1}{}}
\tikzoption{out max distance}{\tikz@to@set@distances{}{}{}{#1}}
\tikzoption{out distance}{\tikz@to@set@distances{}{}{#1}{#1}}
\tikzoption{min distance}{\tikz@to@set@distances{#1}{}{#1}{}}
\tikzoption{max distance}{\tikz@to@set@distances{}{#1}{}{#1}}
\tikzoption{distance}{\tikz@to@set@distances{#1}{#1}{#1}{#1}}

\def\tikz@to@set@distances#1#2#3#4{%
  \tikz@to@setifnotempy{#1}{\tikz@to@in@min}{\let\tikz@to@end@compute=\tikz@to@end@compute@looseness}%
  \tikz@to@setifnotempy{#2}{\tikz@to@in@max}{\let\tikz@to@end@compute=\tikz@to@end@compute@looseness}%
  \tikz@to@setifnotempy{#3}{\tikz@to@out@min}{\let\tikz@to@start@compute=\tikz@to@start@compute@looseness}%
  \tikz@to@setifnotempy{#4}{\tikz@to@out@max}{\let\tikz@to@start@compute=\tikz@to@start@compute@looseness}%
  \tikz@to@switch@on%
}

\def\tikz@to@setifnotempy#1#2#3{%
  \def\pgf@temp{#1}%
  \ifx\pgf@temp\pgfutil@empty\else\def#2{#1}#3\fi%
}

\def\tikz@to@set@in@looseness#1{%
  \def\tikz@to@in@looseness{#1}%
  \let\tikz@to@end@compute=\tikz@to@end@compute@looseness%
  \tikz@to@switch@on%
}
\def\tikz@to@set@out@looseness#1{%
  \def\tikz@to@out@looseness{#1}%
  \let\tikz@to@start@compute=\tikz@to@start@compute@looseness%
  \tikz@to@switch@on%
}

\def\tikz@to@parse@controls#1and#2\pgf@stop{\tikz@to@set@in@control{#2}\tikz@to@set@out@control{#1}}

\def\tikz@to@set@in@control#1{%
  \def\tikz@to@in@control{#1}%
  \let\tikz@to@end@compute=\tikz@to@end@compute@control%
  \tikz@to@switch@on%
}
\def\tikz@to@set@out@control#1{%
  \def\tikz@to@out@control{#1}%
  \let\tikz@to@start@compute=\tikz@to@start@compute@control%
  \tikz@to@switch@on%
}

\def\tikz@to@bend{30}

\def\tikz@to@out{45}
\def\tikz@to@in{135}

\def\tikz@to@out@looseness{1}
\def\tikz@to@in@looseness{1}

\def\tikz@to@in@min{0pt}
\def\tikz@to@in@max{10000pt}
\def\tikz@to@out@min{0pt}
\def\tikz@to@out@max{10000pt}

\def\tikz@to@switch@on{\let\tikz@to@path=\tikz@to@curve@path}

\def\tikz@to@curve@path{%
  [every curve to]
  \pgfextra{\iftikz@to@relative\tikz@to@compute@relative\else\tikz@to@compute\fi}
  \tikz@computed@path
  \tikztonodes%
  \pgfextra{\tikz@to@relpfalse\tikz@to@relppfalse}%
}

\def\tikz@to@modify#1#2{%
  \pgfutil@ifundefined{pgf@sh@ns@#1}
  {}%
  {\edef#1{#1.#2}}
}%

\def\tikz@to@compute{%
  \let\tikz@tofrom=\tikztostart%
  \let\tikz@toto=\tikztotarget%
  \tikz@to@modify\tikz@tofrom\tikz@to@out%
  \tikz@to@modify\tikz@toto\tikz@to@in%
  \ifx\tikz@to@start@compute\tikz@to@start@compute@looseness%
    \tikz@to@compute@distance%
  \else%
    \ifx\tikz@from@start@compute\tikz@to@start@compute@looseness%
      \tikz@to@compute@distance%
    \fi%
  \fi%
  \tikz@to@start@compute%
  \tikz@to@end@compute%
  \iftikz@to@relp
    \edef\tikz@computed@path{.. controls \tikz@computed@start and \tikz@computed@end .. +(\tikz@toto)}
  \else
    \iftikz@to@relpp  
      \edef\tikz@computed@path{.. controls \tikz@computed@start and \tikz@computed@end .. ++(\tikz@toto)}
    \else
      \edef\tikz@computed@path{.. controls \tikz@computed@start and \tikz@computed@end .. (\tikz@toto)}
    \fi
  \fi
}

\def\tikz@to@compute@distance{\tikz@scan@one@point\tikz@@to@compute@distance(\tikz@tofrom)}
\def\tikz@@to@compute@distance#1{%
  \def\tikz@first@point{#1}%
  \iftikz@to@relp%
    \tikz@scan@one@point\tikz@@@to@compute@distance([shift={(\tikz@toto)}]\tikz@tofrom)%
  \else%
    \iftikz@to@relpp%
      \tikz@scan@one@point\tikz@@@to@compute@distance([shift={(\tikz@toto)}]\tikz@tofrom)%
    \else%
      \tikz@scan@one@point\tikz@@@to@compute@distance(\tikz@toto)%
    \fi%
  \fi}
\def\tikz@@@to@compute@distance#1{%
  \def\tikz@second@point{#1}%
  \tikz@to@compute@distance@main%
}
\def\tikz@to@compute@distance@main{%
  \pgf@process{\pgfpointdiff{\tikz@first@point}{\tikz@second@point}}%
  \ifdim\pgf@x<0pt\pgf@xa=-\pgf@x\else\pgf@xa=\pgf@x\fi%
  \ifdim\pgf@y<0pt\pgf@ya=-\pgf@y\else\pgf@ya=\pgf@y\fi%
  \pgf@process{\pgfpointnormalised{\pgfqpoint{\pgf@xa}{\pgf@ya}}}%
  \ifdim\pgf@x>\pgf@y%
    \c@pgf@counta=\pgf@x%
    \ifnum\c@pgf@counta=0\relax%
    \else%
      \divide\c@pgf@counta by 255\relax%
      \pgf@xa=16\pgf@xa\relax%
      \divide\pgf@xa by\c@pgf@counta%
      \pgf@xa=16\pgf@xa\relax%
    \fi%
  \else%
    \c@pgf@counta=\pgf@y%
    \ifnum\c@pgf@counta=0\relax%
    \else%
      \divide\c@pgf@counta by 255\relax%
      \pgf@ya=16\pgf@ya\relax%
      \divide\pgf@ya by\c@pgf@counta%
      \pgf@xa=16\pgf@ya\relax%
    \fi%
  \fi%
  \pgf@x=0.3915\pgf@xa%
  \pgf@xa=\tikz@to@out@looseness\pgf@x%
  \pgf@xb=\tikz@to@in@looseness\pgf@x%
  \pgfmathsetlength{\pgf@ya}{\tikz@to@out@min}
  \ifdim\pgf@xa<\pgf@ya%
    \pgf@xa=\pgf@ya%
  \fi%
  \pgfmathsetlength{\pgf@ya}{\tikz@to@out@max}
  \ifdim\pgf@xa>\pgf@ya%
    \pgf@xa=\pgf@ya%
  \fi%
  \pgfmathsetlength{\pgf@ya}{\tikz@to@in@min}
  \ifdim\pgf@xb<\pgf@ya%
    \pgf@xb=\pgf@ya%
  \fi%
  \pgfmathsetlength{\pgf@ya}{\tikz@to@in@max}
  \ifdim\pgf@xb>\pgf@ya%
    \pgf@xb=\pgf@ya%
  \fi%
}

\def\tikz@to@start@compute@looseness{%
  \edef\tikz@computed@start{([shift=(\tikz@to@out:\the\pgf@xa)]\tikz@tofrom)}%
}
\def\tikz@to@end@compute@looseness{%
  \edef\tikz@computed@end{+(\tikz@to@in:\the\pgf@xb)}%
}
\def\tikz@to@start@compute@control{%
  \let\tikz@computed@start=\tikz@to@out@control%
}
\def\tikz@to@end@compute@control{%
  \let\tikz@computed@end=\tikz@to@in@control%
}

\let\tikz@to@start@compute=\tikz@to@start@compute@looseness%
\let\tikz@to@end@compute=\tikz@to@end@compute@looseness%

\def\tikz@to@compute@relative{%
  \tikz@scan@one@point\tikz@@to@compute@relative(\tikztostart)%
}
\def\tikz@@to@compute@relative#1{%
  \def\tikz@tofrom{#1}%
  \tikz@scan@one@point\tikz@@@to@compute@relative(\tikztotarget)%
}
\def\tikz@@@to@compute@relative#1{%
  \def\tikz@toto{#1}%
  \begingroup
    \pgfutil@ifundefined{pgf@sh@ns@\tikztostart}
    {%
      \let\tikz@first@point=\tikz@tofrom%
      \let\tikz@tostart@tikz=\pgfutil@empty
    }%
    {%
      {%
        \tikz@tofrom%
        \pgf@xc=\pgf@x%
        \pgf@yc=\pgf@y%
        {%
          \pgftransformreset%
          \pgftransformshift{\pgfqpoint{\pgf@xc}{\pgf@yc}}%
          \pgftransformrotate{\tikz@to@out}%
          \pgftransformshift{\pgfqpoint{-\pgf@xc}{-\pgf@yc}}%
          \pgf@process{\pgfpointtransformed{\tikz@toto}}%
        }%
        \pgf@xc=\pgf@x%
        \pgf@yc=\pgf@y%
        \pgfpointshapeborder{\tikztostart}{\pgfqpoint{\pgf@xc}{\pgf@yc}}%
        \xdef\tikz@tofrom@smuggle{\noexpand\pgfqpoint{\the\pgf@x}{\the\pgf@y}}
      }%
      \let\tikz@first@point=\tikz@tofrom@smuggle%
      \tikz@first@point%
      \edef\tikz@tostart@tikz{(\the\pgf@x,\the\pgf@y)}%
    }%
    \pgfutil@ifundefined{pgf@sh@ns@\tikztotarget}
    {%
      \let\tikz@second@point=\tikz@toto%
    }%
    {%
      {%
        \tikz@toto%
        \pgf@xc=\pgf@x%
        \pgf@yc=\pgf@y%
        {%
          \pgftransformreset%
          \pgftransformshift{\pgfqpoint{\pgf@xc}{\pgf@yc}}%
          \pgftransformrotate{180}%
          \pgftransformrotate{\tikz@to@in}%
          \pgftransformshift{\pgfqpoint{-\pgf@xc}{-\pgf@yc}}%
          \pgf@process{\pgfpointtransformed{\tikz@tofrom}}%
        }%
        \pgf@xc=\pgf@x%
        \pgf@yc=\pgf@y%
        \pgfpointshapeborder{\tikztotarget}{\pgfqpoint{\pgf@xc}{\pgf@yc}}%
        \xdef\tikz@toto@smuggle{\noexpand\pgfqpoint{\the\pgf@x}{\the\pgf@y}}
      }%
      \let\tikz@second@point=\tikz@toto@smuggle%
    }%
    \tikz@second@point%
    \edef\tikz@totarget@tikz{(\the\pgf@x,\the\pgf@y)}%
    \tikz@to@compute@distance@main%
    \edef\tikz@to@first@distance{\the\pgf@xa}%
    \edef\tikz@to@second@distance{\the\pgf@xb}%
    \pgftransformreset%
    \pgf@process{\tikz@first@point}%
    \pgf@xa=\pgf@x%
    \pgf@ya=\pgf@y%
    \pgf@process{\tikz@second@point}%
    \advance\pgf@x by-\pgf@xa%
    \advance\pgf@y by-\pgf@ya%
    \pgfpointnormalised{}%
    \pgf@xc=\pgf@x%
    \pgf@yc=\pgf@y%
    \pgf@xb=-\pgf@x%
    \pgf@yb=-\pgf@y%
    {%
      \pgftransformshift{\tikz@first@point}%
      \pgftransformcm{\pgf@sys@tonumber\pgf@xc}{\pgf@sys@tonumber\pgf@yc}{\pgf@sys@tonumber\pgf@yb}{\pgf@sys@tonumber\pgf@xc}%
                      {\pgfpointorigin}%
      \pgf@process{\pgfpointtransformed{\pgfpointpolar{\tikz@to@out}{\tikz@to@first@distance}}}%
      \xdef\tikz@computed@start{(\the\pgf@x,\the\pgf@y)}%
    }
    {%
      \pgftransformshift{\tikz@second@point}%
      \pgftransformcm{\pgf@sys@tonumber\pgf@xc}{\pgf@sys@tonumber\pgf@yc}{\pgf@sys@tonumber\pgf@yb}{\pgf@sys@tonumber\pgf@xc}%
                      {\pgfpointorigin}%
      \pgf@process{\pgfpointtransformed{\pgfpointpolar{\tikz@to@in}{\tikz@to@second@distance}}}%
      \xdef\tikz@computed@end{(\the\pgf@x,\the\pgf@y)}%
    }
    \xdef\tikz@computed@path{
      \tikz@tostart@tikz
      .. controls \tikz@computed@start and \tikz@computed@end ..
      \tikz@totarget@tikz}%
  \endgroup
}

\title{Traces in symmetric monoidal categories}

\author{Kate Ponto and Michael Shulman}

\thanks{Both authors were supported by National Science Foundation
 postdoctoral fellowships during the writing of this paper.}

\date{Version of \today}

\usepackage{braket}
\let\setof\Set
\renewcommand{\Set}{\ensuremath{\mathbf{Set}}}

\newcommand{\fix}{F}

\theoremstyle{definition}

  \makeatletter\let\c@conv\c@thm\makeatother
  \numberwithin{conv}{section}

\begin{document}
\maketitle

\tableofcontents

\section*{Introduction}
\label{sec:introduction}

The original notion of \emph{trace} is, of course, the trace of a
square matrix with entries in a field.  An important and far-reaching
categorical generalization of this notion applies to any endomorphism
of a \emph{dualizable object} in a symmetric monoidal category;
see~\cite{dp:duality,kl:cpt}.

This generalization has a number of applications which are often 
closely connected with the study of fixed points.
One application of particular importance is the Lefschetz fixed point 
theorem and its variants and generalizations, many of which can be 
deduced directly from the naturality and functoriality of the canonical 
symmetric monoidal trace.

The purpose of this expository note is to describe this
notion of trace in a symmetric monoidal category, along with its
important properties (including naturality and functoriality), and to give as
many examples as possible.
Among other things, this note is intended as background for~\cite{PS2}
and~\cite{PS3}, in which the symmetric monoidal trace is generalized
to the context of bicategories and indexed monoidal categories, and \cite{kate:traces, kate:rel, equiv, PS4}, which give applications of the 
bicategorical trace to fixed point theory.

In \S\ref{sec:fixed-points} we describe one way to understand the connection between
traces and fixed points.  This provides motivation for the 
formal definitions in \S\ref{sec:traces}.  In \S\ref{sec:examples} we give many examples of
the trace.  These include topological examples connected to the Lefschetz
fixed point theorem and its generalizations as well as examples arising 
in other contexts.  In \S\ref{sec:twist-trac-transf}  we define a generalization
of the trace from \S\ref{sec:traces}.  This trace arises in many applications
and is a generalization of the classical transfer.  Then in \S\ref{sec:prop-trace-smc} we describe
``coherence'' properties of the trace, while in \S\ref{sec:funct-smc} we describe its functoriality and naturality,
including the Lefschetz fixed point theorem as an application.
Finally, in \S\ref{sec:vistas} we remark on some generalizations.

\section{Traces and fixed points}
\label{sec:fixed-points}

A common feature of all the examples we will consider is that \emph{traces give information about fixed points}.
Thus, before embarking on formalities, in this section we will attempt to give some intuition for why this should be so.

Suppose we are working in a monoidal category, and consider a morphism whose source and target are tensor products, such as $f\colon A\otimes B\otimes C \to D\otimes B$.
We think of such an $f$ as a ``process'' which takes three inputs, of types $A$, $B$, and $C$, and produces two outputs, of types $D$ and $B$.
In keeping with this intuition, we draw $f$ as follows:
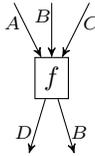
\begin{tikzcenter}
  \node[draw,rectangle] (f) at (0,0) {$f$};
  \draw[->] (f) -- node [ed,near end] {$B$} +(0.3,-1);
  \draw[->] (f) -- node [ed,near end,swap] {$D$} +(-0.3,-1);
  \draw[<-] (f) -- node [ed,near end,swap] {$C$} +(0.5,1);
  \draw[<-] (f) -- node [ed,near end] {$A$} +(-0.5,1);
  \draw[<-] (f) -- node [ed,near end] {$B$} +(0,1);
\end{tikzcenter}
This is an example of \emph{string diagram notation} for monoidal
categories, which is ``Poincar\'e dual'' to the usual sort of
diagrams: instead of drawing objects as \emph{vertices} and morphisms
as \emph{arrows} connecting these vertices, we draw objects as
\emph{arrows} and morphisms as \emph{vertices}, often with boxes
around them.
See~\cite{penrose:negdimten,js:geom-tenscalc-i,jsv:traced-moncat,selinger:graphical}
for more about string diagram calculus.  In particular, we note that
Joyal and Street~\cite{js:geom-tenscalc-i} proved that the ``value''
of a string diagram is invariant under deformations of diagrams, so
that we can prove theorems by topological reasoning; see
\autoref{thm:smc-ut-cyclicity} for an example.

If the source and target of the morphism $f$ above are the same, then a \emph{fixed point} of $f$ is a morphism 
$f^\dagger\colon \ast\to X$ (where $\ast$ denotes the unit for the monoidal structure) such that
\begin{tikzcenter}
  \node[draw,rectangle] (f) at (0,0) {$f$};
  \draw[->] (f) -- node [ed,near end] {$X$} +(0,-1);
  \node [draw,rectangle] (fd) at (0, 1.2) {$f^\dagger$};
  \draw[<-] (f) -- node [ed,swap] {$X$} (fd);
  \draw[<-] (fd) -- node [ed,near end] {$\ast$} +(0,1.2);
  \node at (1.5,0.5) {$=$};
  \node[draw,rectangle] (fd') at (2.5,0.7) {$f^\dagger$};
  \draw[->] (fd') -- node [ed,near end] {$X$} +(0,-1.5);
  \draw[<-] (fd') -- node [ed,near end] {$\ast$} +(0,1.5);
\end{tikzcenter}

We will be interested in traces of more general morphisms.  For these we will need to 
be able to 
duplicate inputs and outputs, which we draw as follows:
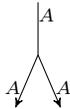
\begin{tikzcenter}
  \draw[->] (0.5,1.2) -- node [ed,near start] {$A$} (0.5,0.5)
  -- node [ed,near end,swap] {$A$} (0.2,-0.2);
  \draw[->] (0.5,0.5) -- node [ed,near end] {$A$} (0.8,-0.2);
\end{tikzcenter}
This is only possible if our monoidal category is \emph{cartesian} monoidal, in which case the above duplication process is the diagonal $\Delta\colon A\to A\times A$.

Now suppose only \emph{one} of the inputs of $f$ matches its output:

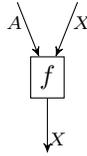
\begin{tikzcenter}
  \node[draw,rectangle] (f) at (0,0) {$f$};
  \draw[->] (f) -- node [ed,near end] {$X$} +(0,-1);
  \draw[<-] (f) -- node [ed,near end] {$A$} +(-0.4,1);
  \draw[<-] (f) -- node [ed,near end,swap] {$X$} +(0.4,1);
\end{tikzcenter}
An \emph{($A$-parametrized) fixed point} of $f$ is   
a morphism $f^\dagger\colon A\to X$ such that
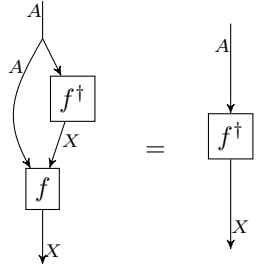
\begin{tikzcenter}
  \node[draw,rectangle] (f) at (0,0) {$f$};
  \draw[->] (f) -- node [ed,near end] {$X$} +(0,-1);
  \draw[<-] (f) to[out=120,in=-120,rel] node [ed,near end] {$A$} (0,2) coordinate (delta);
  \path (f) +(0.4,1.2) node[draw,rectangle] (fd) {$f^\dagger$};
  \draw[<-] (f) -- node [ed,near end,swap] {$X$} (fd);
  \draw[<-] (fd) -- (delta) -- node [ed,near end] {$A$} +(0,0.5);
  \node at (1.5,0.5) {$=$};
  \node[draw,rectangle] (fd') at (2.5,0.7) {$f^\dagger$};
  \draw[->] (fd') -- node [ed,near end] {$X$} +(0,-1.5);
  \draw[<-] (fd') -- node [ed,near end] {$A$} +(0,1.5);
\end{tikzcenter}
i.e.\ ``$f(a,f^\dagger(a))=f^\dagger(a)$ for any $a\in A$'' or 
``$f^\dagger(a)\in X$ is a fixed point of $f(a,-)$ for any $a\in A$.''

A common way to look for fixed points in concrete situations is by iteration: we start with some $x_0\in X$ and compute $x_1 = f(a,x_0)$, $x_2 = f(a,x_1)$, and so on.
If ever $x_{n+1}=x_n$, we've found a fixed point.
But even if not, we can hope that the sequence $(x_0,x_1,x_2,\dots)$ will ``converge'' to a fixed point.
Two contexts where this works are the contraction mapping theorem in topology and the least-fixed-point combinator in domain semantics.

In order to mimic this in abstract language, we need a notion of \emph{feedback}, i.e.\ a way to plug the output of a given morphism into its input.
In diagrammatic terms, given a morphism one of whose inputs matches \emph{one} of its outputs:
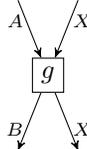
\begin{tikzcenter}
  \node[draw,rectangle] (f) at (0,0) {$g$};
  \draw[->] (f) -- node [ed,near end,swap] {$B$} +(-0.4,-1);
  \draw[->] (f) -- node [ed,near end] {$X$} +(0.4,-1);
  \draw[<-] (f) -- node [ed,near end] {$A$} +(-0.4,1);
  \draw[<-] (f) -- node [ed,near end,swap] {$X$} +(0.4,1);
\end{tikzcenter}
we want to construct a new morphism in which the $X$ input and output have been ``fed back into each other'' somehow:
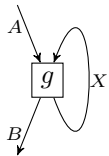
\begin{tikzcenter}
  \node[draw,rectangle] (f) at (0,0) {$g$};
  \draw[->] (f) -- node [ed,near end,swap] {$B$} +(-0.4,-1);
  \draw[<-] (f) -- node [ed,near end] {$A$} +(-0.4,1);
  \draw[<-] (f) to[out=70,in=-70,looseness=10] node [ed] {$X$} (f);
\end{tikzcenter}
This is called a \emph{trace} of the morphism $f$.
In fact, Hyland~\cite{bh:traced-premon} and Hasegawa~\cite{hasegawa} have independently observed the following (see \S\ref{sec:prop-trace-smc}).

\begin{thm}\label{hhtrace}
In a cartesian monoidal category, to give a notion of trace is precisely the same as to give a fixed-point operator which assigns to every morphism $A\times X\to X$ a fixed point $A\to X$ in a coherent way.
\end{thm}

This relationship is especially important in computer science.
However, in topology, we are interested in maps which may have zero, one, or many fixed points.
Thus, we can't expect to have a fixed-point operator acting on the whole category, since there is no way to specify a fixed point for a map which has no fixed points.
Instead, we would like to know, given a map, does it have any fixed points, and if so, how many and what are they?
Thus, we need an operation which produces, instead of a fixed point, some sort of ``invariant'' carrying information about whatever fixed points a map may have.
A fruitful approach to this is to map our cartesian monoidal category \bC\ into a larger category \bD\ in which morphisms can be ``superimposed'' or ``added.''
Then we may hope for a trace or a fixed-point operator in \bD\ which computes the ``sum'' of all the fixed points that a map may have (or ``zero'' if it has none).

Often our functor $Z\colon \bC\to\bD$ will be like the free abelian group functor, and the ``fixed point'' of $Z(f)$ will be something like $\sum_{f(a) = a} a$.
And just as the free abelian group functor maps cartesian products not to cartesian products, but to tensor products, if we want $Z$ to be a monoidal functor, we usually cannot expect \bD\ to be cartesian monoidal, only symmetric monoidal.
Therefore, a trace in \bD\ no longer implies a fixed point operator.
However, since \bC\ is cartesian, we still have diagonal morphisms for objects in the image of $Z$, and this is all we really need.

In fact, if we can find a category \bD\ which is suitably ``additive,'' then it often comes with a canonical notion of trace for free.
The idea is to split the ``feedback'' diagram into a composition of three pieces:
\begin{center}
  \begin{tikzpicture}
    \node[draw,rectangle] (f) at (0,0) {$g$};
    \draw (f) -- node [ed,near end,swap] {$B$} +(-0.4,-1);
    \draw (f) -- node [ed,near end] {$A$} +(-0.4,1);
    \draw (f) to[out=70,in=-70,looseness=15] node [ed] {$M$} (f);
  \end{tikzpicture}
  \qquad\raisebox{2.8cm}{$=$}\qquad
  \begin{tikzpicture}
    \node[draw,rectangle] (f) at (0,0) {$g$};
    \draw (f) -- node [ed,near end,swap] {$B$} +(-0.4,-1);
    \draw (f) -- node [ed,near end] {$A$} +(-0.4,1);
    \draw (f) to[out=70,in=-70,looseness=15] node [ed] {$M$} (f);
    \fill[white] (-1,0.4) rectangle (1,0.6);
    \fill[white] (-1,-0.4) rectangle (1,-0.6);
  \end{tikzpicture}
\end{center}
The morphism
\begin{tikzcenter}
  \draw (0,0) to[out=90,in=90,looseness=3]
  node[ed,very near start] {$M$}
  node[ed,very near end] {$M$}
  (1,0);
\end{tikzcenter}
(from the unit object to $M\ten M$) is called a ``coevaluation'' or ``unit.''
If $M$ is of the form $Z(X)$ for some $X\in \bC$, then the coevaluation is supposed to pick out a formal sum such as $\sum_{x\in X} x\ten x$.
(To be precise, the second string labeled $M$ is actually its ``dual'' \(\rdual{M}\), and so the sum is actually $\sum_{x\in X} x\ten \rdual{x}$.)
Similarly, the morphism
\begin{tikzcenter}
  \draw (0,0) to[out=-90,in=-90,looseness=3]
  node[ed,swap,very near start] {$M$}
  node[ed,swap,very near end] {$M$}
  (1,0);
\end{tikzcenter}
is called an ``evaluation'' or ``counit.''
For $M=Z(X)$, the evaluation is supposed to be supported on pairs of the form $x\ten x$ (or, more precisely, $\rdual{x}\ten x$), and to give zero when applied to a pair $x\ten x'$ for $x\neq x'$.
If this is the case, then the composite
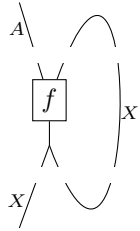
\begin{tikzcenter}
  \node[draw,rectangle] (f) at (0,0) {$f$};
  \draw (f) -- node [ed,near end] {$A$} +(-0.4,1.3);
  \draw (f) -- ++(0,-0.6) coordinate (delta)
  -- node [ed,near end,swap] {$X$} +(-0.4,-1.1);
  \draw (delta) to[out=-70,in=70,looseness=10] node [ed,swap] {$X$} (f);
  \fill[white] (-1,0.5) rectangle (2,0.7);
  \fill[white] (-1,-0.9) rectangle (2,-1.1);
\end{tikzcenter}
will act as follows:
\begin{equation*}
  a \mapsto \sum_{x\in X} a\ten x\ten x
  \mapsto \sum_{x\in X} f(a,x)\ten x
  \mapsto \sum_{x\in X} f(a,x)\ten f(a,x)\ten x
  \mapsto \sum_{x\in X\atop f(a,x)=x} f(a,x).
\end{equation*}
Thus, as desired, it picks out the sum of all the fixed points of $f$.

Traces constructed in this way from evaluation and coevaluation maps are called \emph{canonical} traces.
In~\cite{jsv:traced-moncat}, Joyal and Street showed that any traced monoidal category \bD\ can be embedded in a larger one $\mathrm{Int}(\bD)$, in such a way that the given traces in \bD\ are identified with canonical traces in $\mathrm{Int}(\bD)$.
Therefore, for the purposes of finding fixed-point invariants, there is no loss in restricting our attention to canonical traces.

However, the choice of a particular \bD\ does restrict the maps for which we can calculate our fixed-point invariant, since the resulting objects in \bD\ must admit evaluations and coevaluations.
This property is called being \emph{dualizable}.
For instance, in the free abelian group on $X$, the sum $\sum_{x\in X} x\ten x$ is only defined when $X$ is finite.
A given functor $\bC\to\bD$ thus induces a notion of ``finiteness'' on objects of \bC.
We will see in \S\ref{sec:examples} that different choices of \bD\ can drastically affect the notion of finiteness, as well as the utility and computability of traces.
However, in most applications the choice of \bD\ is straightforward, and the resulting finiteness restriction not onerous.

\section{Traces}
\label{sec:traces}

We now move on to the abstract study of canonical traces in symmetric
monoidal categories; in the next section we will specialize to a
number of examples and see how we obtain information about fixed
points.  We begin with the formal definition of dualizability.

Let \bC\ be a symmetric monoidal category with product $\ten$ and unit
object $I$.
We will omit the associativity and unit isomorphisms of \bC\ from the notation (effectively pretending that \bC\ is strict, as is allowable by the coherence theorem), and we write \fs\ for any instance or composite of instances of the symmetry isomorphism of \bC.

\begin{defn}
  An object $M$ of \bC\ is \textbf{dualizable} if there exists an
  object $\rdual{M}$, called its \textbf{dual}, and maps
  \begin{align*}
    I &\too[\eta] M\ten \rdual{M} & \rdual{M}\ten M &\too[\ep] I
  \end{align*}
  satisfying the triangle identities
  \[(\id_M\ten \ep)(\eta\ten \id_M) = \id_M \quad\text{and}\quad
  (\ep\ten \id_\rdual{M})(\id_\rdual{M}\ten \eta) = \id_{\rdual{M}}.\]
\end{defn}
 We call $\ep$ the \textbf{evaluation} and $\eta$
  the \textbf{coevaluation}; some authors call them the \emph{counit}
  and the \emph{unit}.  We say that \bC\ is \textbf{compact closed}
  if every object is dualizable.

As suggested in \S\ref{sec:fixed-points}, dual pairs are
represented graphically by turning around the direction of arrows; see
\autoref{fig:moncat-dual}.
Note that the unit object $I$ is represented by the lack of any strings, such as in the input to $\eta$ and the output of $\varepsilon$.
Strictly speaking, there should be boxes at the ends of these ``caps'' and ``cups'' labeled by $\eta$ and $\varepsilon$ respectively, but these labels are almost universally omitted in string diagram notation (this is also justified by a theorem of Joyal and Street).
The triangle identities for a dual
pair translate graphically as ``bent strings can be straightened;''
see \autoref{fig:triangle}.

\begin{figure}[tb]
  \centering
    \begin{tabular}{m{30mm}m{20mm}m{40mm}m{20mm}}
    \begin{center}Coevaluation\\$\eta\maps I \to M\ten \rdual{M}$\end{center} &
    \begin{center}
      \begin{tikzpicture}
        \draw[<-] (0,0) to[out=90,in=180] node [ed,near start] {$M$} (0.5,1)
        to[out=0,in=90] node [ed,near end] {$\rdual{M}$} (1,0);
      \end{tikzpicture}
    \end{center}
    &
    \begin{center}Evaluation\\$\ep\maps \rdual{M}\ten M\to I$\end{center} &
      \begin{tikzpicture}
        \draw[<-] (0,0) to[out=-90,in=180] node [ed,near start,swap] {$\rdual{M}$} (0.5,-1)
        to[out=0,in=-90] node [ed,near end,swap] {$M$} (1,0);
      \end{tikzpicture}
  \end{tabular}
  \caption{Coevaluation and evaluation}
  \label{fig:moncat-dual}
\end{figure}
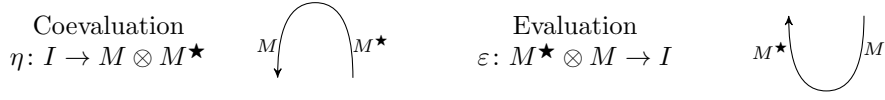 

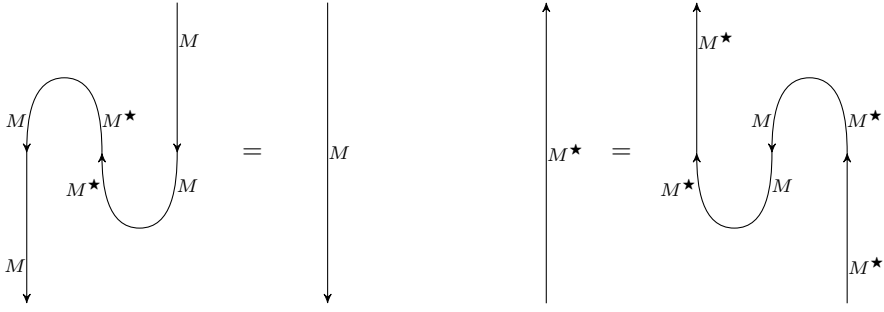
\begin{figure}[tb]
  \centering
  \begin{tabular}{m{50mm}m{15mm}m{50mm}}
    \begin{tikzpicture}
      \draw[<-] (0,0) -- node[ed,near start] {$M$} ++(0,2) coordinate (a);
      \draw[<-] (a) to[rrel,out=90,in=180] node [ed,near start] {$M$} (0.5,1)
      to[rrel,out=0,in=90] node [ed,near end] {$\rdual{M}$} (0.5,-1) coordinate (b);
      \draw[<-] (b) to[rrel,out=-90,in=180] node [ed,near start,swap] {$\rdual{M}$} (0.5,-1)
      to[rrel,out=0,in=-90] node [ed,near end,swap] {$M$} (0.5,1) coordinate (c);
      \draw[<-] (c) -- node[ed,near end,swap] {$M$} ++(0,2);
      \node at (3,2) {$=$};
      \draw[<-] (4,0) -- node[ed,swap] {$M$} ++(0,4);
    \end{tikzpicture} &
    &
    \begin{tikzpicture}
      \draw[->] (-2,0) -- node[ed,swap] {$\rdual{M}$} ++(0,4);
      \node at (-1,2) {$=$};
      \draw[<-] (0,4) -- node[ed,near start] {$\rdual{M}$} ++(0,-2) coordinate (a);
      \draw[<-] (a) to[rrel,out=-90,in=180] node [ed,near start,swap] {$\rdual{M}$} (0.5,-1)
      to[rrel,out=0,in=-90] node [ed,near end,swap] {$M$} (0.5,1) coordinate (b);
      \draw[<-] (b) to[rrel,out=90,in=180] node [ed,near start] {$M$} (0.5,1)
      to[rrel,out=0,in=90] node [ed,near end] {$\rdual{M}$} (0.5,-1) coordinate (c);
      \draw[<-] (c) -- node[ed,near end] {$\rdual{M}$} ++(0,-2);
    \end{tikzpicture}
  \end{tabular}
  \caption{The triangle identities}
  \label{fig:triangle}
\end{figure}
Any two duals of an object $M$ are isomorphic; an isomorphism can be
constructed from $\eta$ and $\ep$.  If $\rdual{M}$ is a dual of $M$,
then $M$ is a dual of $\rdual{M}$.  And if $M$ and $N$ are dualizable,
any map $f\maps Q\ten M\to N\ten P$ has a \textbf{dual} or
\textbf{mate}
\[\rdual{f}\maps \rdual{N}\ten Q\to P\ten \rdual{M},\]
given by the composite
\[\rdual{N}\ten Q \xto{\id\ten\id\ten \eta} \rdual{N}\ten Q\ten M \ten \rdual{M}
\xto{\id\ten f \ten \id} \rdual{N} \ten N\ten P \ten \rdual{M}
\xto{\ep \ten \id\ten \id} P\ten \rdual{M}.
\]
In particular, if $M$ is dualizable, any endomorphism $f\maps M\to M$
has a dual $\rdual{f}\maps \rdual{M}\to \rdual{M}$.

There are various equivalent characterizations of dualizable objects.
For example, when \bC\ is closed, $M$ is dualizable if and only if the
canonical map \[M\ten \Hom(M,I) \to \Hom(M,M)\] is an isomorphism.
However, for us the above definition is most appropriate.

We now move on to the simplest form of trace.

\begin{defn}\label{def:comm-trace}
  Let \bC\ be a symmetric monoidal category, $M$ a dualizable object
  of \bC\, and $f\maps M\to M$ an endomorphism of $M$.  The
  \textbf{trace} of $f$, denoted $\tr(f)$, is the following composite:
  \begin{equation}
    I \too[\eta] M\ten \rdual{M} \too[f\ten \id] M\ten \rdual{M}
    \xiso{\fs} \rdual{M}\ten M \too[\ep] I.
    \label{eq:comm-trace}
  \end{equation}
  The \textbf{Euler characteristic} of a dualizable $M$ is the trace
  of its identity map.
\end{defn} 

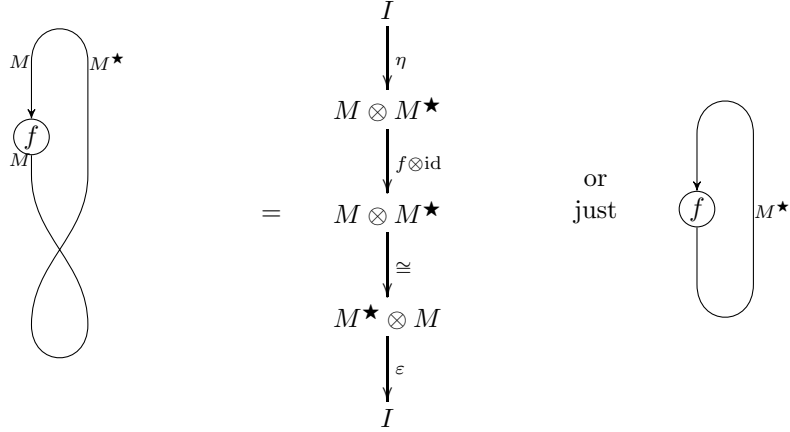
\begin{figure}[tb]
  \centering
  \begin{tabular}{m{30mm}m{5mm}m{25mm}m{15mm}m{20mm}}

  \begin{tikzpicture}
   \node[vert](f) at (0, -1.5){$f$};
   
   \draw[->](f)--node[ed,swap, near start]{$M$}(0, -2)to [out=-90, in =90](.75, -4)to [out=-90, in =-90, looseness=2] (0,-4)
      to [out=90, in =-90, looseness=1] (.75, -2)--node[ed,swap, at end]{$\rdual{M}$}(.75, -.5) to [out=90, in =90, looseness=2] (0, -.5)--
     node[ed,swap, at start]{$M$}(f);

\end{tikzpicture}

    & $=$ &
    $\xymatrix{I \ar[d]^{\eta} \\ M\ten \rdual{M} \ar[d]^{f\ten
        \id} \\ M \ten \rdual{M} \ar[d]^{\iso} \\ \rdual{M}\ten
      M\ar[d]^{\ep} \\ I}$
    & \begin{center}or\\just\end{center} &
  \begin{tikzpicture}
   \node[vert](f) at (0, -1){$f$};
   
   \draw[->](f)--(0, -2)to [out=-90, in =-90, looseness=2] (.75,-2)
    --node[ed,swap]{$\rdual{M}$}(.75, 0) to [out=90, in =90, looseness=2] (0, 0)--(f);

\end{tikzpicture}
  \end{tabular}
  \caption{The trace}
  \label{fig:ut-smctrace}
\end{figure}

The trace of a morphism translates as graphically as ``feeding its output into
its input;'' see \autoref{fig:ut-smctrace}.

References for this notion of trace
include~\cite{dp:duality,fy:brd-cpt,js:brd-tensor,jsv:traced-moncat,kl:cpt}.
It is an endomorphism of the unit object $I$, and
does not depend on the choice of dual for $M$ or on the choice of the
maps $\eta$ and $\ep$.  It also has the following fundamental property.

\begin{prop}[Cyclicity]\label{thm:smc-ut-cyclicity}
  For any $f\maps M\to N$ and $g\maps N\to M$ with $M,N$ both
  dualizable, we have $\tr(fg)=\tr(gf)$.
\end{prop}
\begin{proof}
  The following proof is really only rendered comprehensible by string
  diagram notation (see \autoref{fig:smc-ut-cyclicity}
).
  \begin{multline*}
    \scriptstyle\tr(fg) = \ep \fs (fg \ten \id) \eta = \ep\fs(f\ten
    \id)(g\ten \id)\eta =
    \ep\fs(f\ten \id)(\id\ten\ep\ten\id)(\eta\ten\id\ten\id)(g\ten \id)\eta =\\
    \scriptstyle\ep\fs(\id\ten\ep\ten\id)(f\ten
    \id\ten\id\ten\id)(\id\ten\id\ten g\ten \id)(\eta\ten \eta) =
    (\ep\ten\ep)\fs(\id\ten\id\ten g\ten \id)(f\ten \id\ten\id\ten\id)(\eta\ten \eta) =\\
    \scriptstyle(\ep\ten\ep)\fs(\id\ten\id\ten g\ten
    \id)(\id\ten\id\ten \eta)(f\ten \id)\eta =
    \ep\fs(\id\ten\ep\ten\id)(g\ten\id\ten\id\ten\id)(\eta\ten\id\ten\id)(f\ten\id)\eta =\\
    \scriptstyle\ep\fs(g\ten\id)(\id\ten\ep\ten\id)(\eta\ten\id\ten\id)(f\ten\id)\eta
    = \ep\fs(g\ten\id)(f\ten\id)\eta = \ep\fs(gf\ten\id)\eta = \tr(gf)\qedhere
  \end{multline*}
\end{proof}

\begin{figure}
  \begin{tabular}{m{10mm}m{2mm}m{22mm}m{2mm}m{22mm}m{2mm}m{19mm}m{0mm}m{10mm}}
  \begin{tikzpicture}[scale=.8]
   \node[vert](f) at (0, -1){$f$};
   \node[vert2](g) at (0, -3){$g$};
   
   \draw[->](f)--node[ed,swap]{$N$}(g);

   \draw[->](g)to [out=-90, in =90] (.75, -5)to [out=-90, in =-90, looseness=2] (-.0,-5)
      to [out=90, in =-90, looseness=1] 
     (.75,-3.25)--node[ed,swap, at end]{$\rdual{M}$}(.75, 0) to [out=90, in =90, looseness=2] (0, 0)--
     node[ed,swap, at start]{$M$}(f);

\end{tikzpicture}
&=&

  \begin{tikzpicture}[scale=.8]
   \node[vert](f) at (0, -1){$f$};
   \node[vert2](g) at (-1.5, -3){$g$};
   
   \draw[->](f)--node[ed]{$N$}(0, -2) to [out=-90, in =-90, looseness=2](-.75, -2) node[ed, right]{$\rdual{N}$}
     to [out=90, in =90, looseness=2](-1.5,-2)--node[ed, swap, at start]{$N$} (g);

   \draw[->](g)--node[ed, swap]{$M$}(-1.5, -3.5) to [out=-90, in =90](0, -5)to [out=-90, in =-90, looseness=2] (-.75,-5)
      to [out=90, in =-90, looseness=1]
     (.75,-3.5)--node[ed, swap, at end]{$\rdual{M}$}(.75, 0) to [out=90, in =90, looseness=2] (0, 0)--
     node[ed, at start, swap]{$M$}(f);

\end{tikzpicture}
&=&

\begin{tikzpicture}[scale=.8]
   \node[vert](f) at (0, -3){$f$};
   \node[vert2](g) at (-1.5, -1){$g$};
   
   \draw[->](f) --node[ed, swap]{$N$} (0, -3.5)  to [out=-90, in =-90, looseness=2](-.75, -3.5) --
       node[ed, right, at end]{$\rdual{N}$}(-.75, 0)to [out=90, in =90, looseness=2](-1.5,0)--node[ed, swap, at start]{$N$}  (g);

   \draw[->](g)--node[ed, swap, near start]{$M$}(-1.5, -3.5)to [out=-90, in =90](0, -5)to [out=-90, in =-90, looseness=2] (-.75,-5)
      to [out=90, in =-90, looseness=1](.75,-3.5)--node[ed, swap, at end]{$\rdual{M}$}(.75, -2) to [out=90, in =90, looseness=2] (0, -2)--
            node[ed, at start, swap]{$M$}(f);

\end{tikzpicture} 
&=&

  \begin{tikzpicture}[scale=.8]
   \node[vert](f) at (-3, -4){$f$};
   \node[vert2](g) at (-1.5, -2){$g$};
   
   \draw[->](f) --node[ed, swap]{$N$}(-3, -5) to [out=-90, in =-90, looseness=2](-3.75, -5) --
     (-3.75, -1.5) to[out=90, in =-90, looseness=.8]
       node[ed, right, at end]{$\rdual{N}$} (-2.25, 0)to [out=90, in =90, looseness=2](-3,0)to[out=-90, in =90, looseness=.8] 
node[ed, swap, at start]{$N$}(-1.5, -1.5)-- (g);

   \draw [->](g)--node[ed, near end]{$M$}(-1.5, -3)to [out=-90, in =-90, looseness=2] (-2.25,-3)
     to [out=90, in =90]node[ed, swap, at end]{$\rdual{M}$}(-2.25, -3) to [out=90, in =90, looseness=2] (-3, -3)--
       node[ed, at start, swap]{$M$}(f);

\end{tikzpicture}
&=&
  \begin{tikzpicture}[scale=.8]
   \node[vert](f) at (-1.5, -4){$f$};
   \node[vert2](g) at (-1.5, -2){$g$};
   
   \draw[->](f) --(-1.5, -5) to [out=-90, in =-90, looseness=2](-2.25, -5) --
     (-2.25, -1.5) to[out=90, in =-90, looseness=.8]
       node[ed, right, at end]{$\rdual{N}$} (-1.5, 0)to [out=90, in =90, looseness=2](-2.25,0)to[out=-90, in =90, looseness=.8] 
node[ed, swap, at start]{$N$}(-1.5, -1.5)-- (g);

   \draw [->](g)--node[ed]{$M$}(f);

\end{tikzpicture}
\end{tabular}
\caption{Cyclicity of the trace}
  \label{fig:smc-ut-cyclicity}
\end{figure}
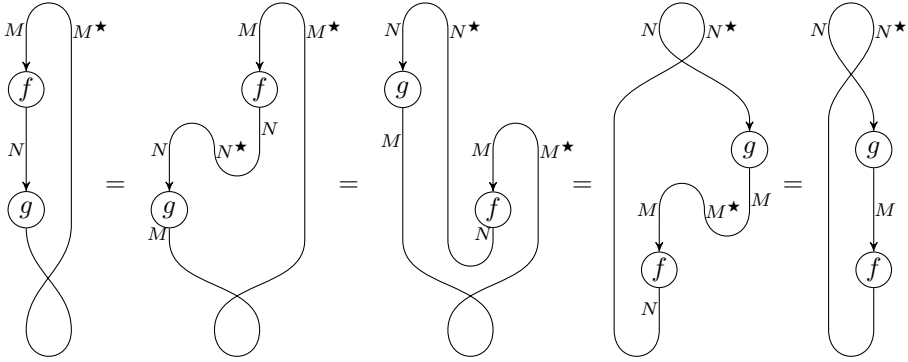

We will consider additional properties of the 
trace in \S\ref{sec:prop-trace-smc}.

\section{Examples of traces}
\label{sec:examples}

\begin{eg}
  Let $\bC=\Vect{k}$ be the category of vector spaces over a
  field $k$.  A vector space is dualizable if and only if it is
  finite-dimensional, and its dual is the usual dual vector space.  We
  have $I=k$ and $\bC(I,I)\iso k$ by multiplication.  Using this identification,
  \autoref{def:comm-trace} recovers the usual trace of a matrix.  The
  Euler characteristic of a vector space is its dimension.
\end{eg}

\begin{eg}
  Let $\bC=\bMod{R}$ be the category of modules over a
  commutative ring $R$.  The dualizable objects are the finitely
  generated projectives.  As in $\Vect{k}$, we have $I=R$ and
  $\bC(R,R)\iso R$, so every endomorphism of a finitely generated projective module has a
  trace which is an element of $R$.  The Euler
  characteristic of such a module is its rank, regarded as an element
  of $R$ (so that, for instance, the Euler characteristic of a
  rank-$p$ free $(\bbZ/p)$-module is zero).
\end{eg}

\begin{eg}
  Again, let $R$ be a commutative ring and consider the category
  $\bCh{R}$ of chain complexes of $R$-modules, with its symmetric
  monoidal tensor product.  The ``correct'' symmetry isomorphism
  $M\ten N \iso N\ten M$ introduces a sign: $a\ten b \mapsto
  (-1)^{|a||b|}(b\ten a)$.  The dualizable objects are again the
  finitely generated projectives, the unit is again $R$ itself (in
  degree 0), and endomorphisms of the unit can again be identified
  with elements of $R$.  The trace of an endomorphism of a finitely
  generated projective chain complex, called its \emph{Lefschetz number}, is
  the alternating sum of its degreewise traces.  Likewise, the Euler
  characteristic of such a chain complex is the alternating sum of its
  ranks.  This generalizes straightforwardly to modules over a DGA.
\end{eg}

\begin{eg}
  There is also a symmetric monoidal category $\Ho(\bCh{R})$,
  called the \emph{derived category} of $R$, obtained from
  $\bCh{R}$ by formally inverting the quasi-isomorphisms
  (morphisms which induce isomorphisms on all homology groups).  The
  dualizable objects in $\Ho(\bCh{R})$ are those that are
  quasi-isomorphic to an object that is dualizable in $\bCh{R}$, and
  the two kinds of traces also agree.
\end{eg}

\begin{eg}\label{eg:ncob-smc}
  Let $\bCob{n}$ be the category whose objects are closed
  $(n-1)$-dimensional manifolds, and whose morphisms are
  diffeomorphism classes of cobordisms.  Composition is by gluing,
  cylinders $M\times [0,1]$ give identities, and disjoint union
  supplies a symmetric monoidal structure.  The unit object is the
  empty set $\emptyset$, and an endomorphism of $\emptyset$ is just a
  closed $n$-manifold.

  Every object of $\bCob{n}$ is dualizable: the evaluation and
  coevaluation are both $M\times [0,1]$, regarded either as a
  cobordism from $\emptyset$ to $M\sqcup M$ or from $M\sqcup M$ to
  $\emptyset$.  The trace of a cobordism from $M$ to $M$ is
  the closed $n$-manifold obtained by gluing the two components of
  its boundary together.  In particular, the Euler characteristic of a
  closed $(n-1)$-manifold $M$ is $M\times S^1$, regarded as a
  cobordism from $\emptyset$ to itself.
\end{eg}

\begin{eg}\label{eg:cart-dual}
  In a \emph{cartesian} monoidal category, 
  the only dualizable object is the terminal object.  Thus, in this
  case there are no interesting traces.  However, as suggested in \S\ref{sec:fixed-points}, often we can obtain
  useful dualities and traces by applying a functor from such a
  category to a non-cartesian monoidal category.  Such a functor $F$
  induces a notion of ``finiteness'' on its domain in the evident way:
  $X$ is ``finite'' if $F(X)$ is dualizable.  For an endomorphism $f$ of
  such an $X$, we can then compute the trace of $F(f)$.

  Probably the simplest such functor is the free abelian group
  functor $\bbZ[-]\maps \Set\to\Ab$.  Of course, $\bbZ [X]$ is
  dualizable in \Ab\ if and only if $X$ is a finite set.  If
  $f\maps X\to X$ is an endomorphism of a finite set, then the trace
  of $\bbZ [f]\maps \bbZ [X]\to \bbZ [X]$ is easily seen to be simply the
  number of fixed points of $f$.
  This justifies the hope expressed in \S\ref{sec:fixed-points} that by mapping a cartesian monoidal category into an ``additive'' one, we could extract information about fixed points which may or may not be present.
  The next few examples can also be viewed in this light.
\end{eg}

\begin{eg}\label{eg:stab-duality}
  Suppose that instead of a set we start with a topological space.
  The category \Top\ is, of course, cartesian monoidal, so in order to
  obtain interesting dualities we need to apply a functor landing in
  some non-cartesian monoidal category.  One obvious guess, by analogy
  with \autoref{eg:cart-dual}, would be the category of abelian topological
  groups and the free abelian topological group functor.  It usually
  turns out to be better, however, to use a more refined notion: the
  category \Sp\ of \emph{spectra}.

  For the reader unfamiliar with spectra some intuition can be gained
  as follows.  A connective spectrum can be thought of as analogous to an
  abelian topological group, except that its group structure is only
  associative, unital, and commutative up to homotopy and all higher
  homotopies.  The passage from connective spectra to arbitrary
  spectra is then analogous to the passage from bounded-below chain
  complexes to arbitrary ones.  There is a symmetric monoidal category
  \Sp\ of spectra  and a ``free'' functor
  $\Sip\maps \Top\to\Sp$, usually called the
  \emph{suspension spectrum} functor.
  (Actually, there are many such categories
    \Sp, all equivalent up to homotopy, but each having different
    technical advantages and disadvantages; see for instance~\cite{ekmm,mmss,mm:eqosp-smod}.
    We will generally gloss over such distinctions.)
    The monoidal structure of \Sp\
  is called the ``smash product'' $\wedge$, and its unit object is the
  \emph{sphere spectrum} $S$ (which can be identified with
  $\Sip$ of a point).

  Since \Sp\ is not cartesian monoidal, we can hope for it to have an
  interesting duality theory.  However, it turns out that in \Sp\ it
  is only reasonable to ask for duality up to homotopy.  Thus, instead
  of \Sp\ we usually work with the category $\Ho(\Sp)$ obtained from
  it by inverting the ``stable equivalences;'' this is called the \emph{stable
    homotopy category}.  We still have a functor $\Sip\maps
  \Top\to\Ho(\Sp)$ which factors through $\Ho(\Top)$ (in which we
  invert the weak homotopy equivalences).  The reason for the use of
  ``stable'' is that for compact spaces $M$ and $N$, the homset
  $\Ho(\Sp)(\Sip(M),\Sip(N))$ can be identified
  with the set of stable homotopy classes of maps from $M$ to $N$,
  i.e.\ the colimit over $n$ of the sets of homotopy classes of maps $\Sigma^n(M_+) \to
  \Sigma^n(N_+)$.

  One can now show that if $M$ is a closed smooth manifold, or more
  generally a compact ENR (Euclidean Neighborhood Retract), 
  then $\Sip(M)$ is dualizable in
  $\Ho(\Sp)$; its dual is the Thom spectrum $T\nu$ of its stable
  normal bundle.
  This is proven in \cite{atiyah:thom, dp:duality,lms:equivariant}.
  The set of endomorphisms of the sphere spectrum $S$
  is $\colim_n \pi_n(S^n)$, which is isomorphic to $\mathbb{Z}$;
  thus traces in $\Ho(\Sp)$ can be
  identified with integers.

  Using this identification, the trace of an endomorphism can be
  identified with its \emph{fixed point index}. 
The fixed point index
  is an integer which is defined classically, for a map with isolated
  fixed points, as the sum over all fixed points $x$ of the degree of
  the self-map induced by the ``difference'' of the identity map 
  and the endomorphism on a sufficiently small sphere surrounding $x$.
  This turns out to be homotopy invariant, and so for an arbitrary map
  it can be defined by homotoping to a map with isolated fixed points.
  In particular, this is necessary for the identity map, whose fixed
  point index is the Euler characteristic of the manifold (this is, of
  course, the origin of the term ``Euler characteristic'' for traces
  of identity maps in general).  See \cite{b:lefschetz, d:book} for the classical 
approach to the index and \cite{d:index_enr, d:index,dp:duality}
  for the identification of this trace with the classical fixed point index and Euler characteristic.

 For compact spaces  the induced notions of duality and trace can also be formulated
  without using the stable homotopy category; we replace $S$ with
  the $n$-sphere $S^n$ for some large enough finite $n$.  In this
  guise it is called \emph{$n$-duality}; references
  include~\cite{dp:duality,lms:equivariant}.
\end{eg} 

$\Sp$ and $\bCh{R}$ are two instances of a general phenomenon: a
symmetric monoidal category that has an associated symmetric monoidal
homotopy category.  A general theory of when and how symmetric
monoidal structures descend to homotopy categories is given by the
study of \emph{monoidal model categories}, as
in~\cite[Ch.~4]{hovey:modelcats}.

\begin{eg}\label{eg:eqv-parm-duality}
  For a compact Lie group $G$, there is an \emph{equivariant stable
    homotopy category} $\Ho(\gSp{G})$, which is related to the
  category $\gTop{G}$ of $G$-spaces in the same way that $\Ho(\Sp)$ is
  related to \Top; see for instance~\cite{mm:eqosp-smod}.
  Now the suspension and stabilization take place
  not relative to ordinary spheres $S^n$, but relative to
  representations of $G$.  The category $\Ho(\gSp{G})$
  is also symmetric monoidal and admits
  a suspension $G$-spectrum functor from $G$-spaces.

  The dualizable objects in $\Ho(\gSp{G})$
  include the equivariant suspension spectra of
  closed smooth $G$-manifolds and compact $G$-ENR's.
  Such dual pairs can be also described using $V$-duality
  for a representation $V$.  A reference for equivariant duality is
  \cite{lms:equivariant}.  Traces in $\Ho(\gSp{G})$ are again called
  fixed point indices; see~\cite{td:groups, ulrich, equiv}.
\end{eg}

\begin{eg}\label{eg:parm-duality}
  Another variation is to consider \emph{parametrized} duality, which
  instead of spaces or $G$-spaces starts from spaces \emph{over} a
  base space $B$.  In~\cite{maysig:pht}, May and Sigurdsson construct
  a category $\bEx B$ of \emph{parametrized spectra} over $B$, which
  is symmetric monoidal, has a symmetric monoidal homotopy
  category $\Ho(\bEx{B})$, and admits a
  suspension functor $\Sigma_{B,+}^{\infty}$ from $\Top/B$ that is
  similar to $\Sip$.

  If $M$ is a fibration over $B$, then
  $\Sigma^\infty_{B,+}(M)$ is dualizable in $\Ho(\bEx{B})$ if and only if
  each of its fibers is dualizable in the usual stable homotopy
  category.  In particular, a fibration of closed smooth manifolds gives rise to a
  dualizable parametrized spectrum.  The trace of a fiberwise
  endomorphism is once again called its fixed point index;
  see~\cite{d:index}.
\end{eg}

\begin{rmk}
For parametrized spaces and spectra it is often more
  illuminating to consider a different type of duality called
  \emph{Costenoble-Waner duality}, and its associated notion of trace.
  These notions of duality and trace do not take place in a symmetric
  monoidal category, but rather in a bicategory arising from an
  indexed symmetric monoidal category; see~\cite[Ch.~18]{maysig:pht}
  and~\cite{kate:traces, PS2,PS3}.
\end{rmk}

Here are some more ``toy'' examples.

\begin{eg}\label{eg:rel}
  Let $\bRel$ be the category whose objects are sets and whose
  morphisms from $M$ to $N$ are relations $R\subset M\times N$; we
  write $R\maps M\hto N$ to avoid confusion with functions $M\to N$.
  If $S\maps N\hto P$ is another relation, their composite is
  \[S R = \Big\{(m,p) \;\Big|\; \exists n \text{ with }
  (m,n)\in R \text{ and } (n,p)\in S\Big\}.
  \]
  A symmetric monoidal structure on $\bRel$ is induced by the
  cartesian product of sets (which is \emph{not} the cartesian product
  in \bRel).

  There is a functor $\Set\to\bRel$ which is the identity on objects
  and takes a function $f\maps X\to Y$ to its graph $\Gamma_f =
  \setof{(x,f(x)) | x\in X }$.  Moreover, \emph{every} set is
  dualizable in $\bRel$, and moreover is its own dual; the relations
  \eta\ and \ep\ are both the identity relation on $X$, considered as
  a relation $*\hto X\times X$ or $X\times X\hto *$, respectively.
  Thus every set is ``finite'' relative to this functor.  However, the
  trade off is that traces contain correspondingly less information, since
  the only endomorphisms of the unit $*$ in $\bRel$ are the empty
  relation and the full one.  If we regard these as the \emph{truth
    values} ``false'' and ``true,'' respectively, then the trace of a relation
  $R\maps M\hto M$ is the truth value of the statement ``$\exists
  m: (m,m)\in R$.''  In particular, for a function $f\maps M\to M$, 
  $\tr(\Gamma_f)$ is true if and only if $f$ has a fixed point.

  This example can be generalized to internal relations in
  any suitably well-behaved category.
\end{eg}

\begin{eg}\label{eg:sup}
  Let $\Sup$ denote the category of \emph{suplattices}: that
  is, its objects are posets with all suprema and its morphisms are
  supremum-preserving maps.  (Of course, a suplattice also has all
  infima, but suplattice maps need not preserve infima.)  There is a
  tensor product of suplattices, concisely described by saying that
  suplattice maps $M\ten N\to P$ represent functions $M\times N\to P$
  which preserve suprema in each variable separately.  The unit object
  is the suplattice $I = (0\le 1)$.

  We can see an analogy between \Sup\ and \Ab\ by regarding suprema in
  a suplattice as similar to sums in an abelian group.  For instance,
  there is a ``free suplattice'' functor $\Set\to\Sup$ which simply
  takes a set $A$ to its power set $\calP(A)$; the ``free generators''
  are the singleton sets, and a subset $B\subseteq A$ is the ``formal
  sum'' $\sum_{x\in B} \{x\}$.  Every such
  power set is dualizable, so every set is ``finite'' relative to the
  functor $\calP$.  Explicitly, we have $\rdual{\calP(A)} \iso \calP
  (A)$ with coevaluation
  \[\eta(1) = \bigvee_{a\in A} \{a\}\ten \{a\}\]
  and evaluation
  \[\ep(X\ten Y) =
  \begin{cases}
    1 & \text{if } X\cap Y \neq \emptyset\\
    0 & \text{otherwise}.
  \end{cases}
  \]
  Note that a suplattice map $f\maps \calP (A) \to \calP (A)$ is equivalent to
  a function $A\to \calP (A)$, and thereby to a relation $R_f\subset
  A\times A$.  The trace of such a map in $\Sup$ is easily
  verified to be $\id_I$ if there is an $a\in A$ with $(a,a)\in R_f$
  and $0$ otherwise; thus we essentially recapture \autoref{eg:rel}.

  However, not all dualizable suplattices are power sets.  For
  instance, if $A$ is an Alexandrov topological space (one where
  arbitrary intersections of open sets are open), then its open-set
  lattice $\calO (A)$ is a dualizable suplattice.
  In this case, for a continuous map $f\maps A\to A$, the trace of
  $f^{-1}\maps \calO (A)\to \calO (A)$ is $\id_I$ if there is an $a\in A$
  such that $a\le f(a)$ in the specialization order, and $0$
  otherwise.  (Recall that the \emph{specialization order} of a
  topological space is defined so that $x\le y$ if and only if every open set
  containing $y$ also contains $x$.)
\end{eg}

\begin{eg}
  Continuing the analogy between \Sup\ and \Ab, it is natural to
  consider commutative monoid objects in \Sup\ as analogous to
  commutative rings.  A commutative monoid object in a symmetric
  monoidal category \bC\ is an object $R$ with morphisms $R\ten R\to
  R$ and $I\to R$ satisfying evident axioms.
  In particular, for any topological space $B$, the open-set lattice
  $\calO(B)$ is a commutative monoid object in $\Sup$ whose
  ``multiplication'' map is intersection $\cap\maps
  \calO(B)\ten\calO(B)\to\calO(B)$.  Likewise, for any continuous
  $f\maps A\to B$, we have a monoid homomorphism $f\inv\maps
  \calO(B)\to\calO(A)$.  In this way a category of suitably nice
  topological spaces can be identified with the opposite of a
  subcategory of commutative monoids in \Sup; see~\cite{jt:galois}
  and~\cite[Ch.~C1]{ptj:elephant2}.
  This is analogous to how the
  category of affine schemes can be identified with the opposite of
  the category of commutative rings.

  If \bC\ has coequalizers preserved
  on both sides by $\ten$, then for any commutative monoid object $R$ the category
  of $R$-modules in \bC\ is itself symmetric
  monoidal under the tensor product given by the usual coequalizer
  \[M\ten R\ten N \toto M\ten N \to M\ten_R N.\]
  In particular, this applies to
  $\calO(B)$-modules in \Sup\ for any space $B$.
  Given any continuous map $p\maps X\to
  B$, $\calO(X)$ becomes a $\calO(B)$-algebra (that is, there is a
  monoid homomorphism $p\inv\maps \calO(B)\to\calO(X)$) and thus a
  $\calO(B)$-module.  If $p$ is a local homeomorphism (that is, $X$ is
  the ``espace etale'' of a sheaf over $B$), then $\calO(X)$ is a
  dualizable $\calO(B)$-module that is its own dual: the evaluation is
  \begin{align*}
    \calO(X)\ten_{\calO(B)}\calO(X) &\too[\ep]\calO(B)\\
    (W,W') &\mapsto p(W\cap W')
  \end{align*}
  and the coevaluation is
  \begin{align*}
    \calO(B) &\too[\eta] \calO(X)\ten_{\calO(B)}\calO(X)\\
    U &\mapsto \bigvee_{p(W)\subseteq U \atop  \mathclap{p|_W \text{is a homeomorphism}}}    W \ten W.
  \end{align*}
  The unit $\calO(B)$-module is $\calO(B)$, and an $\calO(B)$-module
  map $\calO(B)\to\calO(B)$ is determined uniquely by where it sends
  $B\in \calO(B)$ (the unit for $\cap$).  Any map $f\maps X\to X$ over
  $B$ gives a map $f\inv\maps \calO(X)\to\calO(X)$ of
  $\calO(B)$-modules, whose trace is characterized by
  \[B\mapsto \setof{ b\in B |  \exists x\in p\inv(b) : f(x)=x };
  \]
  that is, the set of $b\in B$ such that $f|_{p\inv(b)}$ has a fixed
  point.  In particular, the Euler characteristic of a sheaf $X$ is
  its \emph{support} $p(X)\subseteq B$.  When $B$ is the one-point
  space, $X$ must be discrete, and we recapture power sets in
  $\Sup$.  For more on this point of view,
  see~\cite{rr:shvs-modules}.
\end{eg}

We have so far considered only \emph{symmetric} monoidal categories,
but the definitions we have given make sense with only a braiding, and
there are interesting examples which are not symmetric.

\begin{eg}
  Let $\mathbf{Tang}$ be the category of \emph{tangles}: its objects
  are natural numbers $0,1,2,\dots$, and its morphisms from $n$ to $m$
  are tangles from $n$ points to $m$ points.  A tangle is like a
  braid, except that strings can be turned around, so that $n$ need
  not equal $m$; see~\cite{fy:brd-cpt}.  $\mathbf{Tang}$ is braided
  but not symmetric monoidal; its product is disjoint union and its
  unit object is $0$.  Every object is its own dual, the endomorphisms
  of the unit are links, and the trace of an endo-tangle is its
  ``tangle closure'' into a link.  There are oriented and framed
  variants.
\end{eg}

However, the trace as we have defined it is not quite correct in the
non-symmetric case.  For instance, with our definitions, the trace of
the identity $\id_2$ in $\mathbf{Tang}$ is two \emph{linked} circles,
while the trace of the braiding $\fs_2$ is two \emph{unlinked}
circles; clearly it would make more sense for this to be the other way
round.  This can be remedied with the notion of a \emph{balanced}
monoidal category, which is a braided monoidal category in which each
object is equipped with a ``double-twist'' automorphism;
see~\cite{js:brd-tensor,jsv:traced-moncat}.  Symmetric monoidal
categories can be identified with balanced ones in which every
double-twist is the identity.  In a balanced monoidal category, we
define the trace of an endomorphism $f$ by including a double-twist
with $f$ in between \eta\ and \ep; this remedies the problem noted
above with $\mathbf{Tang}$.  For simplicity, however, in this paper we
will consider only the symmetric case.

\section{Twisted traces and transfers}
\label{sec:twist-trac-transf}

The examples in \S\ref{sec:examples} show that the canonical trace defined in \S\ref{sec:traces} does give useful information about fixed points, but usually it only indicates their presence or absence, or at best counts the number of them (with multiplicity).
However, in \S\ref{sec:fixed-points} we saw that in the presence of ``diagonals'', we could hope to extract not merely the number of fixed points, but the fixed points themselves.
The use of diagonals in this way turns out to be a special case of the following more general notion of ``twisted trace.''

\begin{defn}\label{def:twisted-comm-trace}
  Let \bC\ be a symmetric monoidal category, $M$ a dualizable object
  of $\bC$, and $f\maps Q\ten M\to M\ten P$ a morphism in $\bC$.  The
  \textbf{trace} $\tr(f)$ of $f$ is the following composite:
  \begin{equation}
    Q \too[\eta] Q\ten M\ten \rdual{M} \too[f] M \ten P \ten \rdual{M}
    \xiso{\fs} \rdual{M}\ten M\ten P \too[\ep] P
    \label{eq:twisted-comm-trace}
  \end{equation}
\end{defn}

This trace is displayed graphically in Figure~\ref{fig:smctrace}.

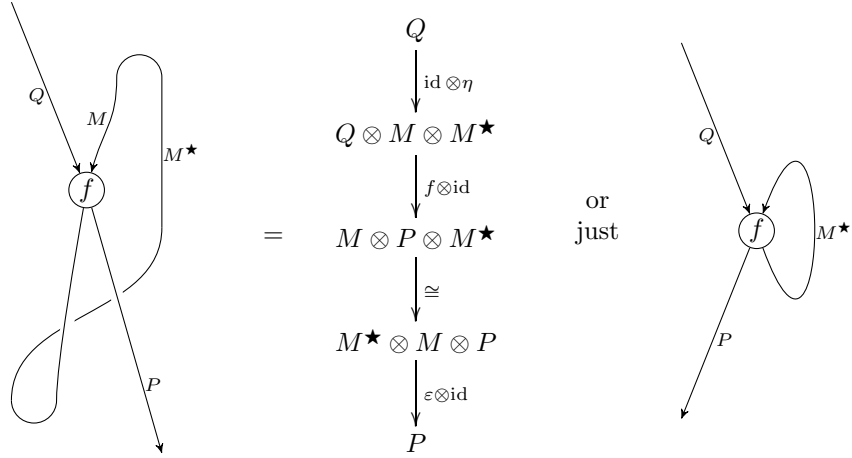
\begin{figure}[tb]
  \centering
  \begin{tabular}{m{30mm}m{5mm}m{25mm}m{15mm}m{20mm}}
    \begin{tikzpicture}
      \node[vert] (f) at (0,0) {$f$};
      \draw[<-] (f) -- node [ed] {$Q$} +(-1,2.5);
      \draw[->] (f) -- node [ed,near end] {$P$} +(1,-3.5) coordinate (p);
      \draw[<-] (f) to[rrel,out=75,in=-90] node [ed] {$M$} (0.4,1.5)
      arc (180:0:0.3)
      -- node [ed] {$\rdual{M}$} ++(0,-2) node[coordinate] (c) {};
      \begin{pgfonlayer}{background}
        \draw (c) to[rrel,out=-90,in=90] (-2,-2.3) node[coordinate] (d) {};
      \end{pgfonlayer}
      \draw (d) arc (-180:0:0.3) node[coordinate] (e) {}
      to[out=90,in=-100,looseness=0.5] (f);
      \begin{pgfonlayer}{background}
        \draw[white,line width=5pt] (f) -- (p);
        \draw[white,line width=5pt] (f) -- (e);
      \end{pgfonlayer}
    \end{tikzpicture}
    & $=$ &
    $\xymatrix{Q \ar[d]^{\id\ten \eta} \\ Q\ten M\ten \rdual{M} \ar[d]^{f\ten
        \id} \\ M\ten P \ten \rdual{M} \ar[d]^{\iso} \\ \rdual{M}\ten
      M\ten P \ar[d]^{\ep\ten \id} \\ P}$
    & \begin{center}or\\just\end{center} &
    \begin{tikzpicture}
      \node[vert] (f) at (0,0) {$f$};
      \draw[<-] (f) -- node [ed] {$Q$} +(-1,2.5);
      \draw[->] (f) -- node [ed] {$P$} +(-1,-2.5);
      \draw[->] (f) to[out=-70,in=70,looseness=15]
      node [ed,swap] {$\rdual{M}$} (f);
    \end{tikzpicture}
  \end{tabular}
  \caption{The ``twisted''  trace}
  \label{fig:smctrace}
\end{figure}

\begin{rmk}
Since \bC\ is symmetric, it may seem odd to write the domain of $f$ as $Q\ten M$ but its codomain as $M\ten P$.
Indeed, in the literature on symmetric monoidal traces it is more common to align the $M$'s on one side, as in the right-hand version of Figure~\ref{fig:smctrace}.
Our notation is chosen instead to match that of the bicategorical generalization presented in~\cite{PS2}, in which case the order we have chosen here is the only possibility.
\end{rmk}

Of course, when $Q=P=I$ is the unit object, this reduces to the previous notion of trace.
It is also cyclic, in a suitable sense.

\begin{lem}\label{thm:twisted-comm-cyclic}
  If $M$ and $N$ are dualizable and $f\maps Q\otimes M\to N\otimes P$
  and $g\maps K\otimes N\to M\otimes L$ are morphisms, then
  \[\tr\Big((g\otimes \id_P)(\id_K \otimes f)\Big)
  = \tr\Big(\fs(f\otimes \id_L) (\id_Q\otimes g)\fs\Big).
  \]
\end{lem}
The string diagram for this lemma is \autoref{fig:smc-cyclicity}, which should be compared to \autoref{fig:smc-ut-cyclicity}.
\begin{figure}[tb]
  \centering
  \begin{tikzpicture}
    \node[vert] (f) at (0,0) {$f$};
    \draw[<-] (f) -- node [ed] {$Q$} +(-1,2);
    \draw[->] (f) -- node [ed,near end] {$P$} +(1.5,-4.5);
    \node[vert2] (g) at (-0.5,-1.5) {$g$};
    \draw[<-] (g) -- node [ed] {$K$} +(-1.3,3.5);
    \draw[->] (g) -- node [ed,near end] {$L$} +(1,-3);
    \draw[->] (f) -- node [ed] {$N$} (g);
    \draw[<-] (f) to[rrel,out=75,in=-90] node [ed] {$M$} (0.4,1.5)
    arc (180:0:0.3)
    -- node [ed] {$\rdual{M}$} ++(0,-3)
    to[rrel,out=-90,in=90] (-3,-2.3)
    coordinate[label={[ed]left:$\rdual{M}$}] arc (-180:0:0.3)
    coordinate[label={[ed]right:$M$}] to[out=90,in=-100] (g);
  \end{tikzpicture}
  \qquad\raisebox{3cm}{=}\qquad
  \begin{tikzpicture}
    \node[vert2] (g) at (0,-0.3) {$g$};
    \draw[<-] (g) to[rrel,out=110,in=-80] node [ed,swap] {$K$} (-1.5,2.2);
    \draw[->] (g) to[rrel,out=-70,in=90] node [ed,near end] {$L$} (1,-2.5)
    to[rrel,out=-90,in=80] (-1,-1.5);
    \node[vert] (f) at (-0.5,-1.5) {$f$};
    \draw[<-] (f) to[rrel,out=110,in=-90] node [ed] {$Q$} (-1,2)
    to[rrel,out=90,in=-110] (0.75,1.4);
    \draw[->] (f) to[rrel,out=-70,in=110] node [ed,swap] {$P$} (0.5,-1.5)
    to[rrel,out=-70,in=100] (1,-1.3);
    \draw[->] (g) -- node [ed] {$M$} (f);
    \draw[<-] (g) to[rrel,out=75,in=-90] node [ed] {$N$} (0.4,1)
    arc (180:0:0.3)
    -- node [ed] {$\rdual{N}$} ++(0,-2) 
    to[rrel,out=-90,in=90] (-3,-1.8)
    coordinate[label={[ed]left:$\rdual{N}$}] arc (-180:0:0.3)
    coordinate[label={[ed]right:$N$}] to[out=90,in=-100] (f);
  \end{tikzpicture}
  \caption{Cyclicity of the ``twisted'' trace}
  \label{fig:smc-cyclicity}
\end{figure}
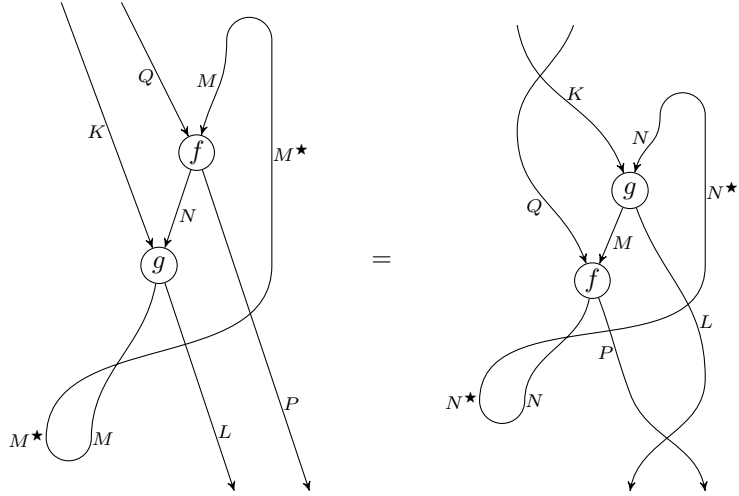

As promised, the trace using a ``diagonal morphism'' is a special case of the general notion of twisting.

\begin{defn}\label{defn:transfer}
  Let $M\in\bC$ be a dualizable object equipped with a ``diagonal''
  morphism $\Delta\maps M\to M\ten M$, and let $f\maps M\to M$ be an
  endomorphism of $M$.  Then the \textbf{trace of $f$
    with respect to $\Delta$} is the trace of $\Delta\circ f\maps M\to
  M\ten M$.  The trace of $\id_M$ with respect to $\Delta$ is called
  the \textbf{transfer} of $M$.
\end{defn}

The trace of $f$ with respect to $\Delta$ is a morphism $I\to M$; by
cyclicity it is also equal to the trace of $(f\ten \id)\circ \Delta$.

Where do such diagonals come from?  Of course, if \bC\ is
\emph{cartesian} monoidal, then any object $M$ has such a diagonal,
but we have seen (\autoref{eg:cart-dual}) that in this case there are
few dualizable objects.
However, we have also seen that for many of the traces we are interested in, the dualizable objects are in the image of a symmetric monoidal functor whose domain \emph{is} cartesian monoidal, and such a functor preserves the existence of diagonals.
That is, if $X$ is an object of a cartesian monoidal category $\bC$ and $Z$ a symmetric monoidal functor with domain $\bC$, then the diagonal $X\to X\times X$ gives rise to a diagonal $Z(X) \to Z(X) \otimes Z(X)$.
This is usually the source of ``diagonals'' in examples.

\begin{eg}\label{eg:freeabgp}
  Let $\bS=\Set$, $\bC=\Ab$, and $Z$ be the
  free abelian group functor $\bbZ[-]$.  In this case the induced
  diagonal $\bbZ[X]\to \bbZ[X]\ten \bbZ[X]$ sends a generator $x$ to
  $x\ten x$.  If $X$ is finite, so that $\bbZ[X]$ is dualizable, the
  trace of $\bbZ[f]$ with respect to this diagonal is $\sum_{f(x)=x}
  x\in\bbZ[X]$.  In particular, the transfer of $\bbZ[X]$ is
  $\sum_{x\in X} x\in\bbZ[X]$.

  Note that while the ordinary trace of $\bbZ[f]$ records only the
  \emph{number} of fixed points of $f$, its trace with respect to
  $\Delta$ records \emph{what} those fixed points are (as elements of
  $\bbZ[X]$).
\end{eg}

\begin{eg}\label{eg:topspec}  As a more sophisticated version of the previous example, 
  let $\bS=\Top$ and let $\bC=\Ho(\Sp)$ be the stable homotopy category.
  Since $\Top$ is cartesian monoidal and the suspension
  spectrum functor $\Sip$ is strong monoidal, the
  diagonal $M\to M\times M$ of any space induces a diagonal
  \[\Delta \maps \Sip (M)\to \Sip (M) \wedge \Sip (M).\]
  Thus, when $M$ is $n$-dualizable, we can define traces and transfers
  with respect to $\Delta$.  This example is the original use of the
  term \emph{transfer}.  In this case, the transfer of an
  $n$-dualizable space $M$ is a map $S\to \Sip(M)$, which
  is by definition an element of $\pi_0^s(M_+)$, the
  $0^{th}$ stable homotopy group of $M$.  If $M$ is connected, there
  is an isomorphism
  \[\pi_0^s(M_+)\cong H_0(M_+)
  \]
  under which the image of the transfer
  is $\chi(M)$ times the generator of $H_0(M_+)$; see
  \cite[III.8.4]{lms:equivariant}.  More generally, if $f$ is an
  endomorphism of $M$, then the trace of $f$ with respect to $\Delta$
  is the fixed point transfer defined by Dold in \cite{d:transfer}.
  We have the same intuition as for the previous example; while the
  fixed point index only \emph{counts} the fixed points, the fixed
  point transfer \emph{records} them.

  There are also equivariant and parametrized transfers.  For example,
  the Becker-Gotleib transfer~\cite{becker} is the parametrized
  transfer of a fibration with compact manifold fibers.
\end{eg}

\begin{eg}
  Recall from \autoref{eg:rel} that we also have a functor
  $\Set\to\bRel$ which is the identity on objects and sends a function
  $f$ to its graph $\Gamma_f$.  In this case, for any endofunction
  $f\maps M\to M$, the trace of $\Gamma_f$ with respect to
  $\Gamma_\Delta$ is the set of all fixed points of $f$, regarded as a
  relation from $\star$ to $M$.  In particular, the transfer of $M$ is
  $M$ itself so regarded.
\end{eg}

\begin{eg}
  Likewise, if $\Sigma$ is the free suplattice functor $\calP\maps
  \Set\to\Sup$ from \autoref{eg:sup}, then for any endofunction
  $f\maps M\to M$ the trace of $\calP(f)$ with respect to
  $\calP(\Delta)$ is also the set of fixed points of $f$, now regarded
  as an element of $\calP(M)$.
\end{eg}

Twisted traces not arising from diagonals are less common, but they do
occur.

\begin{eg}\label{eg:source-twisting-ab}
  Let $f\maps Q\times M\to M$ be a function between sets, where $M$ is
  finite.  Then the trace of $\bbZ[f]\colon \bbZ[Q] \otimes \bbZ[M]
  \to \bbZ[M]$ in \Ab\ is the homomorphism $\bbZ[Q] \to \bbZ$ which
  sends each generator $q\in Q$ to the number of fixed points of
  $f(q,-)$.

  We can also combine this with a transfer, by considering
  $(f,\mathrm{pr}_2)\colon Q\times M \to M\times M$.  This induces a
  map $\bbZ[Q] \otimes \bbZ[M] \to \bbZ[M] \otimes \bbZ[M]$, whose
  trace $\bbZ[Q] \to \bbZ[M]$ sends each generator $q\in Q$ to the sum
  of the fixed points of $f(q,-)$.

  Similarly, for any set $M$ and any function $f\maps Q\times M\to M$,
  the trace of $\Gamma_f$ in $\mathbf{Rel}$ is the relation from $Q$
  to $*$ defined by
  \[\tr(\Gamma_f) = \setof{q\in Q | f(q,-) \text{ has a fixed point}},\]
  and the trace of the induced relation $Q\times M \to M\times M$ is
  \[ \setof{(q,m) \in Q\times M | m \text{ is a fixed point of } f(q,-) }. \]
\end{eg}

\begin{eg}\label{eg:twisted-top}
If $f\colon Q\times M\rightarrow M$ is a continuous map of topological 
spaces and $\Sigma_+^\infty M$ is dualizable, the trace of $f$ is the 
stable homotopy class of a map 
\[Q\rightarrow S^0.\] 
Using explicit descriptions of the coevaluation and evaluation for 
$\Sigma_+^\infty M$, it is not difficult to verify this stable map
is homotopically trivial if the set 
\[\{(q,m)\in Q\times M|f(q,m)=m\}\] is empty.

Let $\pi\colon Q\times M\rightarrow Q$ be the first coordinate projection.
Note that $f$ determines a fiberwise map $F\colon Q\times M\rightarrow 
Q\times M$ by $F(q, m)=(q,f(q,m))$.  The trace of $F$ as described 
in \autoref{eg:parm-duality} coincides with the trace of $f$ under a
corresponding comparison  
of fiberwise stable endomorphism of the unit object over $Q$ with the 
stable maps from $Q$ to $S^0$.  

This trace is related to the \emph{higher Euler characteristics} in \cite{gn:higher}.
\end{eg}

\begin{eg}
  For any topological space $A$ we have an ``intersection'' morphism
  $\cap\maps \calO(A)\ten\calO(A)\to\calO(A)$ in $\Sup$.  If
  $A$ is moreover Alexandrov, so that $\calO(A)$ is dualizable in
  $\Sup$, then for any $f\maps A\to A$, the trace of $f^{-1}
  \circ \cap$ is the function $\calO(A) \to I$ that takes $U\subset A$
  to $1$ if $U$ contains a point $x$ with $x\le f(x)$ and to $0$
  otherwise.  If we identify a suplattice map $g\maps \calO(A) \to I$
  with the closed subset
  \[\bigcap_{K \text{ closed}\atop g(A\setminus
    K)=0} K
  \]
  (which determines it), then the trace of $f^{-1} \circ \cap$ is
  identified with the closure of
  \[\setof{a\in A | a\le f(a)}.\]

  On the other hand, if $B$ is another Alexandrov space with a map
  $m\maps A\times B\to A$, then we have an induced suplattice map
  $m^{-1}\maps \calO(A) \to \calO(A\times B)\iso
  \calO(A)\ten\calO(B)$.  Its trace is the suplattice map $I\to
  \calO(B)$ which takes 1 to the open set
  \[\setof{b\in B | (\exists a\in A)(a\le m(a,b))}.\]
\end{eg}

\begin{eg}
  If $p\maps X\to B$ is any local homeomorphism, then regarding
  $\calO(X)$ as a $\calO(B)$-module we again have an ``intersection''
  morphism $\cap\maps \calO(X)\ten_{\calO(B)}\calO(X)\to\calO(X)$ (corresponding
  to the diagonal $X\to X\times_B X$).  For any $f\maps X\to X$ over
  $B$, the trace of $f\inv\circ \cap$ is the $\calO(B)$-module map
  $\calO(X)\to\calO(B)$ that sends $V\in\calO(X)$ to
  \[\setof{ b\in B |  \exists x\in p\inv(b) \cap V : f(x)=x }.\]
  In particular, the trace of $\cap$ itself sends each $V\in\calO(X)$
  to its support.

  On the other hand, if $q\maps Y\to B$ is another local homeomorphism
  and $f\maps X\times_B Y\to X$ is a map over $B$, then the trace of
  $f\inv\maps \calO(X) \to \calO(X)\ten_{\calO(B)} \calO(Y)$ is the
  map $\calO(B)\to\calO(Y)$ sending the unit $B$ to the open set
  \[\setof{y\in Y | \exists x\in p\inv(q(y)) : f(x,y)=x}.\]
\end{eg}

\section{Properties of traces}
\label{sec:prop-trace-smc}

In addition to cyclicity, the (twisted) symmetric monoidal trace
satisfies many useful naturality properties which we summarize here.
We omit most proofs, which are straightforward diagram chases
(and are especially easy in string diagram notation).
References
include~\cite{dp:duality,lms:equivariant,may:traces,jsv:traced-moncat}.

We begin with invariance under dualization.
Recall that any $f\maps Q\ten M \to M\ten P$ has a mate $\rdual{f}\colon Q\ten \rdual{M} \to \rdual{M}\ten P$, and since $\rdual{M}$ is also dualizable (with dual $M$), after composing with symmetry isomorphisms on either side we can take the trace of $\rdual{f}$ as well.

\begin{prop}
  If $M$ is dualizable and $f\maps Q\ten M \to M\ten P$ is any
  morphism, then $\tr(f) = \tr(\fs\rdual{f}\fs)$.
\end{prop}

In the simple case of square matrices in $\Ab$ or $\bMod R$, this says that the trace of a matrix is equal to the trace of its transpose.

Next we have a naturality property.

\begin{prop}\label{thm:tightening}
  Let $M$ be dualizable, let $f\maps Q\ten M \to M\ten P$ be a map,
  and let $g\maps Q'\to Q$ and $h\maps P\to P'$ be two maps.  Then
  \[h\circ \tr(f)\circ g = \tr\big((\id_M\ten h)\circ f \circ (g\ten \id_M)\big).\]
  In other words, the function
  \[\tr\maps \bC(Q\ten M, M\ten P) \too \bC(Q,P)\]
  is natural in $Q$ and $P$.
\end{prop}

For example, recall from \autoref{eg:source-twisting-ab} that for a function $f\maps Q\times M \to M$, the trace of the induced map in $\Ab$ is the map $\tr(f)\maps \bbZ[Q] \to \bbZ$ which sends each $q\in Q$ to the number of fixed points of $f(q,-)$.
In this case, naturality in $Q$ means that given $g\maps Q'\to Q$, composing $\tr(f)$ with $\bbZ[g] \maps \bbZ[Q'] \to \bbZ[Q]$ counts the number of fixed points of $f(g(q'),-)$.

\autoref{thm:tightening} also implies that quite generally, traces ``calculate fixed points'', as described informally in \S\ref{sec:fixed-points}.

\begin{cor}[Fixed point property]\label{thm:fixedpoint}
  If $M$ is dualizable, $\Delta\maps M\to M\ten P$ is a map,
  $f\maps M\to M$ is an endomorphism, and $h\maps P\to P$ is such that
  $(f\ten h)\Delta = \Delta f$, then \[h \circ \tr(\Delta f) = \tr(\Delta f).\]
\end{cor}

In the situation discussed after \autoref{defn:transfer} of a symmetric monoidal functor $Z\colon \bC\to \bS$, with $\bC$ a cartesian monoidal category,
for any morphism $f\maps M\to M$ in $\bC$ we have that $(f\times f)\Delta = \Delta f$.
Applying the above corollary in $\bS$ then implies $Z(f)\circ \tr(Z(\Delta f))=\tr(Z(\Delta f))$.

Note that $\tr(Z(\Delta f))$ is a map $I\to Z(M)$ in $\bS$, and not in the original category $\bC$.
However, we can see the connections to fixed points explicitly in a couple examples.
If we consider the example where $\bC=\Set$, $\bS=\Ab$, and $Z$ is 
the free abelian group functor, then $\tr(\Delta f)=\sum_{f(x)=x} x$ as computed in \autoref{eg:freeabgp}, and \autoref{thm:fixedpoint} implies that this element of $\bbZ[M]$ is fixed by $\bbZ[f]$.
The corresponding example $\Sigma\colon \Top\to \Sp$ in \autoref{eg:topspec} behaves similarly.

The next few properties are of more technical interest.
An additional naturality property follows directly from cyclicity.

\begin{prop}\label{thm:sliding}
  Let $M$ and $N$ be dualizable and let $f\maps Q\ten M \to N\ten P$
  and $h\maps N\to M$ be maps.  Then
  \[\tr((h\ten \id_P)f) = \tr(f(\id_Q\ten h)).\]
\end{prop}

In fancier language, this says that the function
\[\tr\maps \bC(Q\ten M, M\ten P) \too \bC(Q,P)\]
is ``extraordinary-natural'' (see~\cite{ek:gen-funct-calc}) in the dualizable object $M$.

We now consider compatibility of traces with the monoidal structure.
Note that the unit $I$ is always dualizable with $\rdual{I}=I$.

\begin{prop}\label{thm:vanishing-1}
  If $f\maps Q\ten I\rightarrow I\ten P$ is a morphism in $\bC$, then
  $\tr(f)=f$ (modulo unit isomorphisms).
\end{prop}

If $M$ and $N$ are dualizable, then so is $M\ten N$, with dual
$\rdual{M}\ten \rdual{N}$.  In this case, if we have a map
\[f\maps Q\ten N\ten M \too M\ten N\ten P,
\]
we can either take the trace of $f\fs$ with respect to $M\ten N$, or
we can first take the trace of $f$ with respect to $M$ and then with
respect to $N$; either way results in a map $Q\to P$.

\begin{prop}\label{thm:vanishing-2}
  In the above situation, we have $\tr(f\fs) = \tr(\tr(f))$.
\end{prop}

Alternatively, we could have two maps $f\maps Q\ten M\to M\ten P$ and $g\maps K\ten
N\to N\ten L$.  

\begin{prop}\label{thm:double-superposing}
  In the above situation, we have
  $\tr\big(\fs(f\ten g)\fs\big) = \tr(f) \ten \tr(g)$.
\end{prop}

Taking $N=I$ we obtain the following.

\begin{cor}\label{thm:superposing}
  If $M$ is dualizable and $f\maps Q\ten M\to M\ten P$ and $g\maps
  K\to L$ are maps, then $\tr(\fs(f\ten g)) = \tr(f) \ten g$.
\end{cor}

On the other hand, it is not hard to show that
\autoref{thm:double-superposing} follows from
\autoref{thm:superposing} together with \autoref{thm:vanishing-2}.

Finally, if $M$ and $N$ are dualizable and we have maps $f\maps Q\ten
M\to M\ten P$ and $g\maps P\ten N \to N\ten K$, then we have the
composite
\[(\id_M\ten g)(f\ten \id_N)\maps Q\ten M\ten N \too M\ten N\ten K.\]
The next result then follows from \autoref{thm:tightening} and
\autoref{thm:vanishing-2}.

\begin{cor}
  In the above situation, we have
  \[\tr\big((\id_M\ten g)(f\ten \id_N)\big) = \tr(g)\circ \tr(f).\]
\end{cor}

In particular, we can apply these results to untwisted traces.  Note
that by the Eckmann-Hilton argument, the two operations $\ten$ and
$\circ$ on $\bC(I,I)$ agree (up to unit isomorphisms) and make it a
commutative monoid.  We thereby obtain the following.

\begin{cor}
  If \bC\ is symmetric monoidal, $M$ and $N$ are dualizable, and
  $f\maps M\to M$ and $g\maps N\to N$ are endomorphisms, then
  \[\tr(f\ten g) = \tr(f)\ten \tr(g) = \tr(f)\circ\tr(g).\]
\end{cor}

One final property of traces that should be mentioned here is the
following.

\begin{prop}\label{thm:yanking}
  If $M$ is dualizable, then the trace of $\id_{M\ten M}\maps M\ten
  M\to M\ten M$ is $\id_M\maps M\to M$.
\end{prop}

In~\cite{jsv:traced-moncat} the above properties were taken to constitute
the following definition.

\begin{defn}\label{defn:traced-moncat}
  A symmetric monoidal category \bC\ is \textbf{traced} if it is
  equipped with functions 
  \begin{equation*}
    \tr\maps \bC(Q\ten M, M\ten P)\to \bC(Q,P)\label{eq:tw-trace-maps}
  \end{equation*}
  satisfying the conclusions of Propositions~\ref{thm:tightening},
  \ref{thm:sliding}, \ref{thm:vanishing-1}, \ref{thm:vanishing-2}, and
  \ref{thm:yanking} as well as \autoref{thm:superposing}.
\end{defn}

Actually,~\cite{jsv:traced-moncat} deals with the more general case of
a balanced monoidal category; we have simplified things by treating
only the symmetric case.
A similar set of axioms is considered in~\cite{maltsiniotis:traces}.

Evidently if \bC\ is compact closed (every object is
dualizable), then it is traced in a canonical (and, in fact, unique)
way.  Conversely, it is shown in~\cite{jsv:traced-moncat} that any
traced symmetric monoidal category can be embedded in a compact closed
one, in a way that identifies the given trace with the canonical trace.
On the other hand, much of the interest of the canonical symmetric
monoidal trace lies in its applicability to particular interesting
dualizable objects in categories where not every object is dualizable.

We end this section by making the connection to fixed-point operators  mentioned in \S\ref{sec:fixed-points} precise.
If $\bC$ is a cartesian monoidal category, a trace as in \autoref{defn:traced-moncat} defines  natural functions 
\[\fix\colon\bC(Q\times M, M)\to \bC(Q, M)\]
where for $f\colon Q\times M\to M$, $\fix(f)$ is defined by
\[\tr(Q\times M\xto{\id\times\Delta}Q\times M\times M
\xto{f\times \id}M\times M).\]
Then the above properties become the following.
\begin{itemize}
\item For $f\colon Q\times M \to M$,
  \[\fix(f)=f\circ (\id_Q\times \fix(f))\circ (\Delta_Q\times \id_M).\]
\item For $f\colon Q\times N\to M$ and $g\colon Q\times M\to N$,
  \[\fix(f\circ (\id_Q\times g)\circ (\Delta_Q\times \id_M))=f\circ (\id_Q\times \fix(f\circ (\id_Q\times f)\circ (\Delta_Q\times \id_M)))\circ (\Delta_Q\times \id_M).\]
\item For $f\colon Q\times Q\times M \to M$,
  \[\fix(f\circ (\id_Q\times \Delta_M))=\fix(\fix(f)).\]
\end{itemize}
These are the conditions that define a \textbf{fixed-point operator}.
Conversely, any fixed-point operator $\fix$ on a cartesian monoidal category defines a trace, where the trace of $f\colon Q\times M\to M\times P$ is
\[Q\xto{\fix(f\circ(\id_A\times \pi))}  P\times M\xto{\mathrm{pr}_1} P.
\]
\autoref{hhtrace} says that these two constructions are inverses: thus the existence of a fixed-point operator on a cartesian monoidal category is equivalent to that of a trace.

\section{Functoriality of traces}
\label{sec:funct-smc}

One of the main advantages of having an abstract formulation of trace
is that disparate notions of trace which all fall into the general
framework can be compared functorially.  In this section we summarize
the relevant results and their applicability in some examples,
including the Lefschetz fixed point theorem.

Recall that a \textbf{lax symmetric monoidal functor} $F\maps \bC\rightarrow \bD$
between symmetric monoidal categories consists of a functor $F$ and
natural transformations
\begin{align*}
  \fc\maps F(M)\ten F(N)&\too F(M\otimes N)\\
  \fii\maps I_{\bD}&\too F(I_{\bC})
\end{align*}
satisfying appropriate coherence axioms.  We say $F$ is
\textbf{normal} if \fii\ is an isomorphism, and \textbf{strong} if
\fc\ and \fii\ are both isomorphisms.

\begin{prop}\label{thm:funct-pres-dual}
  Let $F\maps \bC\rightarrow \bD$ be a normal lax symmetric monoidal
  functor, let $M\in \bC$ be dualizable with dual $\rdual{M}$, and
  assume that $\fc\maps F(M)\otimes F(\rdual{M})\rightarrow
  F(M\otimes \rdual{M})$ is an isomorphism (as it is when $F$ is
  strong).  Then $F(M)$ is dualizable with dual $F(\rdual{M})$.
\end{prop}
\begin{proof}
  Suppose given $M$ with dual $\rdual{M}$ exhibited by \eta\ and \ep.  Then
  the maps
  \[I_{\bD}\too[\fii] F(I_{\bC}) \xto{F(\eta)}
  F(M\ten \rdual{M})\xto{\fc^{-1}} F(M)\ten F(\rdual{M})
  \]
  and
  \[F(\rdual{M})\otimes F(M)\too[\fc]F(\rdual{M}\ten M)\xto{F(\ep)} F(I_{\bC})
  \xto{\fii^{-1}} I_{\bD}
  \]
  show that $F(M)$ is dualizable with dual $F(\rdual{M})$.
\end{proof}

In the above situation, we say that $F$ \textbf{preserves} the dual
$\rdual{M}$ of $M$.  Actually, a slightly weaker condition on $F$
suffices for the above conclusion; see~\cite{dp:frob-monoidal}.

\begin{prop}\label{thm:funct-pres-tr}
  If $F$ preserves the dual $\rdual{M}$ of $M$, and moreover
  $\fc\maps F(P) \ten F(M) \to F(P\ten M)$ is an isomorphism (as it is
  whenever $P=I$ and $F$ is normal), then for any map $f\maps Q\ten M
  \to M\ten P$, we have
  \[F(\tr(f)) = \tr\left(\fc \circ F(f)\circ \fc^{-1}\right).
  \]
  In particular, for an endomorphism $f\maps M\rightarrow M$, we have
  \[F(\tr(f))= \fii \circ \tr(F(f)) \circ \fii^{-1}\]
\end{prop}
\begin{proof}
  Use the dual $F(\rdual{M})$ of $F(M)$ to evaluate the right hand
  side.
\end{proof}

\begin{eg}\label{eg:ext_scalar}
If $R$ and $S$ are commutative rings and $\phi\colon R\rightarrow
S$ is a ring homomorphism, then extension of scalars defines a
strong symmetric monoidal functor $(-\otimes_R S)$ from $R$-modules to $S$-modules.
If $M$ is a dualizable $R$-module and $f\colon Q\otimes M
\rightarrow M\otimes P$ is a map of $R$-modules, 
\autoref{thm:funct-pres-tr} implies $\tr(f\otimes_R S)=\tr(f)\otimes_R S$.
If $Q$ and $P$ are both the ring $R$, then as usual, we can think of the traces
of $f$ and $f\otimes_R S$ as elements of $R$ and $S$, respectively; in this 
case, we have $\tr(f\otimes_R S) = \tr(f)\otimes_R S = \phi (\tr(f))$.
\end{eg}

\begin{eg}\label{eg:homology-monoidal}
  Homology is a normal lax symmetric monoidal functor from $\bCh{R}$ 
  or $\Ho(\bCh{R})$ to the
  category $\bGrMod{R}$ of graded $R$-modules.  The K\"unneth theorem
  implies that the natural transformation
  \[\fc\maps H(M_*)\otimes H(N_*)\to H(M_*\otimes N_*)\] 
  is an isomorphism if $M_p$ and $H_p(M_*)$ are projective for each
  $p$.  When these conditions are satisfied (such as when the ground
  ring $R$ is a field),  \autoref{thm:funct-pres-tr}
  implies \[\tr(H_*(f))= H_*(\tr(f))\] for any map of chain complexes
  $f\maps M_*\rightarrow M_*$.  In other words, the Lefschetz number
  is the same whether it is calculated at the level of chain
  complexes or homology.
\end{eg}

\begin{eg}\label{eg:lefschetz-fp-thm}
  By composing the rational cellular chain complex functor with a
  functorial CW approximation, we obtain a functor $\Top\to
  \bCh{\mathbb{Q}}$.  In fact, this functor factors through $\Sp$ via
  a similar construction of CW spectra, and we have an
  induced functor $\Ho(\Sp)\to\Ho(\bCh{\mathbb{Q}})$ which
  is strong symmetric monoidal.  It follows that the fixed
  point index of a continuous map is equal to the Lefschetz number of the induced map on
  chain complexes.  Combining this with the previous example, and
  using rational coefficients so the K\"unneth theorem holds, we
  obtain the \emph{Lefschetz fixed point theorem}: if $f\maps M\to M$
  is a continuous map, where $\Sip (M)$ is dualizable, and
  $\tr H_*(f,\mathbb{Q}) \neq 0$, then $\tr(f)\neq 0$, and thus $f$ has a
  fixed point.  This example was one of the original motivations for
  the abstract study of traces in~\cite{dp:duality}.
\end{eg}

\begin{eg}
  Generalizing the previous example, if $\Sip (M)$ is
  dualizable in $\Ho(\Sp)$ and $f\colon Q\times M\rightarrow M$ is a continuous map,
  then the trace of $\Sip(f)$ in $\Ho(\Sp)$ is a morphism
  $\Sip(Q)\rightarrow S$ in $\Ho(\Sp)$.  We can then
  take the rational homology of this map to obtain a map $\tr(H_*(f_+))\maps
  H_*(Q_+)\to \bbZ$.  On the other hand, we can also apply rational
  homology before taking the trace; this way we obtain a map
  \[H_*(f_+)\colon H_*(Q_+)\otimes H_*(M_+)\rightarrow H_*(M_+)\]
  whose trace is a map $H_*(Q_+)\rightarrow \bbZ$.
  \autoref{thm:funct-pres-tr} then shows
  \[H_*(\tr f_+))=\tr(H_*(f_+)).\]

  When $Q$ is a point, the set of morphisms $\Sip(Q)\to S$ in $\Ho(\Sp)$
   and the set of
  morphisms $H_0(Q_+)\to \bbZ$ in $\Ab$ can both be identified with \bbZ, so no
  information about traces is lost by passage to rational homology.
  For general $Q$, information is lost, but this is not necessarily a
  bad thing: the set of maps $\Sip(Q)\to S$ can be difficult to
  calculate, while the set of maps $H_0(Q_+)\rightarrow \bbZ$ will
  usually be much easier to describe.
\end{eg}

\begin{eg}\label{eg:tft}
  An \emph{$n$-dimensional topological field
    theory}~\cite{atiyah:tqft} with values in a symmetric monoidal
  category \bC\ (such as $\Vect{k}$) is a strong symmetric
  monoidal functor $Z\maps \bCob{n}\to \bC$.  Since every object
  $M$ of $\bCob{n}$ is dualizable, so is each object $Z(M)$.
  Thus, the trace of a cobordism $B$ from $M$ to itself is mapped to
  an endomorphism of the unit of \bC, which can be regarded as an
  algebraic invariant of $B$ computed by the field theory $Z$.

  If $n=1$, then $\bCob{1}$ is the free symmetric monoidal
  category on a dualizable object; thus a 1-dimensional TFT is just a
  dualizable object.  Likewise, if $n=2$, then $\bCob{2}$ is the
  free symmetric monoidal category on a Frobenius algebra; see, for
  instance,~\cite{jk:frob-2dtqft}.  For a higher-dimensional
  generalization, see~\cite{lurie:tft}.
\end{eg}

Finally, since monoidal categories form not just a category but a
2-category, it is natural to ask also how traces interact with monoidal
natural transformations.
Recall that if $F,G\maps \bC\to\bD$ are lax symmetric monoidal functors, a
\textbf{monoidal natural transformation} is a natural transformation
$\al\maps F\to G$ which is compatible with the monoidal constraints of
$F$ and $G$ in an evident way.

\begin{prop}\label{thm:trans-pres-tr}
  Let $F,G\maps \bC\to\bD$ be normal lax symmetric monoidal functors,
  let $\al\maps F\to G$ be a monoidal natural transformation, let $M$ be
  dualizable in \bC, and assume that $F$ and $G$ preserve its dual
  $\rdual{M}$.  Then $\al_M\maps F(M)\to G(M)$ is an isomorphism, and for
  any $f\maps Q\ten M \to M\ten P$, the square
  \[\xymatrix@C=3pc{F(Q) \ar[rr]^{\tr\left(\fc \circ F(f)\circ \fc^{-1}\right)}
    \ar[d]_{\al_Q} && F(P) \ar[d]^{\al_P}\\
    G(Q) \ar[rr]_{\tr\left(\fc \circ G(f)\circ \fc^{-1}\right)} && G(P)}
  \]
  commutes.  In particular, for an endomorphism $f\maps M\to M$, we
  have
  \[\tr(F(f)) = \tr(G(f)).\]
\end{prop}
\begin{proof}
  Since $F(M)$ and $G(M)$ have duals $F(\rdual{M})$ and $G(\rdual{M})$
  respectively, the morphism $\al_{\rdual{M}}\maps F(\rdual{M}) \to
  G(\rdual{M})$ has a dual $\rdual{(\al_{\rdual{M}})}\maps G(M) \to F(M)$.
  A diagram chase (see~\cite[Prop.~6]{dp:frob-monoidal}) shows that
  this is an inverse to $\al_M$.  Then since $G(f) = \al_M \circ
  (F(f))\circ (\al_M)^{-1}$, cyclicity implies that $\tr(F(f)) = \tr(G(f))$.
\end{proof}

\begin{rmk}
  In particular, if \bC\ is
  compact closed and $F,G\maps \bC\to\bD$ are
  strong monoidal, then any monoidal transformation $F\to G$ is an
  isomorphism.  Thus, when we say
  $\bCob{1}$ is the free symmetric monoidal category on a
  dualizable object, as in \autoref{eg:tft}, we really mean that the
  category of strong monoidal functors $\bCob{1}\to \bD$ is
  equivalent to the \emph{groupoid} of dualizable objects in \bD\ and
  isomorphisms between them.  This also generalizes to higher
  dimensions.
\end{rmk}

\autoref{thm:trans-pres-tr} implies some useful ``comparisons between
comparisons'' of ways to compute traces.

\begin{eg}\label{eg:integralhomologytransformation}
  As in \autoref{eg:homology-monoidal}, since $\mathbb{Q}$ is a field,
  the K\"{u}nneth theorem implies that the functor $H_*(-;\mathbb{Q})$
  from $\Ho(\Sp)$ to the category $\GrVect{\mathbb{Q}}$ of graded
  $\mathbb{Q}$-vector spaces is strong symmetric monoidal.  While
  integral homology $H_*(-;\mathbb{Z})$ is not strong symmetric
  monoidal, the K\"unneth theorem implies that it becomes so if
  we quotient by torsion; thus $H_*(-;\mathbb{Z})/\mathrm{Torsion}$ is
  a strong symmetric monoidal functor from $\Ho(\Sp)$ to
  $\bGrMod{\mathbb{Z}}$.  Hence we can compute Lefschetz numbers using
  integral homology as well, and it is natural to want to compare the
  two results.

  As in \autoref{eg:ext_scalar},
  extension of scalars along the inclusion $\iota\colon 
  \mathbb{Z}\rightarrow \mathbb{Q}$ defines a strong symmetric monoidal functor 
  from $\bGrMod{\mathbb{Z}}$ to $\GrVect{\mathbb{Q}}$.  Thus
  we have two functors
  \[H_*(-;\mathbb{Q}) \qquad\text{and}\qquad
  \big(H_*(-;\mathbb{Z})/\mathrm{Torsion}\big)\otimes \mathbb{Q}\]
  from $\Ho(\Sp)$ to $\GrVect{\mathbb{Q}}$, and
  the same inclusion also defines a natural transformation 
  \[\alpha\colon \big(H_*(-;\mathbb{Z})/\mathrm{Torsion}\big)\otimes \mathbb{Q}
  \longrightarrow H_*(-;\mathbb{Q}).\]
  Therefore, we can combine Propositions \ref{thm:funct-pres-tr} and
  \ref{thm:trans-pres-tr} to compare the Lefschetz numbers
  computed using $H_*(-;\mathbb{Z})$ and $H_*(-;\mathbb{Q})$.

  Explicitly, suppose $\Sip(M)$ is dualizable and $f\maps M\to M$ is an endomorphism 
  in $\Ho(\Sp)$.  Then \autoref{thm:trans-pres-tr} implies that, first
  of all, $\alpha$ is an isomorphism
  \[\big(H_*(M;\mathbb{Z})/\mathrm{Torsion}\big)\otimes \mathbb{Q}
  \iso H_*(M;\mathbb{Q}),\]
  and secondly, the trace of $\big(H_*(f;\mathbb{Z})/\mathrm{Torsion})\otimes \mathbb{Q}$ 
  is the same as the trace of $H_*(f;\mathbb{Q})$.
  Since this trace is not twisted, the observation at the end of \autoref{eg:ext_scalar}
  implies
  \[\tr\Big(\big(H_*(f;\mathbb{Z})/\mathrm{Torsion}\big)\otimes \mathbb{Q}\Big)=
  \iota\Big(\tr\big(H_*(f;\mathbb{Z})/\mathrm{Torsion}\big)\Big);\]
  thus the Lefschetz number of $f$ computed using 
  $H_*(-;\mathbb{Z})/\mathrm{Torsion}$ is the same as the Lefschetz number
  computed using $H_*(-;\mathbb{Q})$.
\end{eg}

\section{Vistas}
\label{sec:vistas}

The symmetric monoidal trace described in this paper can
be generalized in various directions.  We have already mentioned its
generalizations to \emph{balanced monoidal categories} (at the end of
\S\ref{sec:examples}) and to \emph{traced monoidal categories}
(\autoref{defn:traced-moncat}).  There are also straightforward
generalizations to symmetric monoidal 2-categories and symmetric
monoidal $n$-categories (modulo a definition of the latter).

Categorifying in a different direction, in~\cite{kate:traces} the
first author introduced a general notion of trace for
\emph{bicategories}, regarded as ``monoidal categories with many
objects''.  This type of trace applies to \emph{noncommutative}
situations such as modules over a noncommutative ring, and was motivated by 
applications to refinements of the Lefschetz fixed point theorem which
use versions of the \emph{Reidemeister trace}.  

This area of topological fixed point theory makes extensive use of comparison results 
like those described in Examples \ref{eg:lefschetz-fp-thm} and \ref{eg:integralhomologytransformation}.  Examples of the classical,
equivariant, and fiberwise results can be found in \cite{W:equiv, N:trace, LR:equiv, H:generalized, GN:trace}
and \cite{b:lefschetz, j:nielsen } contain expository discussion of the classical case.  These results are studied from a 
categorical perspective in \cite{kate:traces, kate:rel, equiv}.

Bicategorical traces
are studied further in~\cite{PS2}, including a suitable notion of
string diagram.
Finally, \cite{PS3} deals with an abstract context
that gives rise to both bicategories and symmetric monoidal categories
(including parametrized spectra as a prime example), and the
relationships of the traces therein.

\bibliographystyle{plain.bst}
\bibliography{traces_2}

\begin{thebibliography}{10}

\bibitem{atiyah:thom}
M.~F. Atiyah.
\newblock Thom complexes.
\newblock {\em Proc. London Math. Soc. (3)}, 11:291--310, 1961.

\bibitem{atiyah:tqft}
M.~F.\ Atiyah.
\newblock Topological quantum field theories.
\newblock {\em Publications Mathematiques de l'IHES}, 8:175--186, 1988.

\bibitem{becker}
J.~C. Becker and D.~H. Gottlieb.
\newblock The transfer map and fiber bundles.
\newblock {\em Topology}, 14:1--12, 1975.

\bibitem{bh:traced-premon}
Nick Benton and Martin Hyland.
\newblock Traced premonoidal categories, 1999.

\bibitem{b:lefschetz}
Robert~F. Brown.
\newblock {\em The {L}efschetz fixed point theorem}.
\newblock Scott, Foresman and Co., Glenview, Ill.-London, 1971.

\bibitem{dp:frob-monoidal}
Brian Day and Craig Pastro.
\newblock Note on {F}robenius monoidal functors.
\newblock {\em New York J. Math.}, 14:733--742, 2008.

\bibitem{d:book}
A.~Dold.
\newblock {\em Lectures on algebraic topology}.
\newblock Springer-Verlag, New York, 1972.
\newblock Die Grundlehren der mathematischen Wissenschaften, Band 200.

\bibitem{d:index_enr}
Albrecht Dold.
\newblock Fixed point index and fixed point theorem for {E}uclidean
  neighborhood retracts.
\newblock {\em Topology}, 4:1--8, 1965.

\bibitem{d:index}
Albrecht Dold.
\newblock The fixed point index of fibre-preserving maps.
\newblock {\em Invent. Math.}, 25:281--297, 1974.

\bibitem{d:transfer}
Albrecht Dold.
\newblock The fixed point transfer of fibre-preserving maps.
\newblock {\em Math. Z.}, 148(3):215--244, 1976.

\bibitem{dp:duality}
Albrecht Dold and Dieter Puppe.
\newblock Duality, trace, and transfer.
\newblock In {\em Proceedings of the International Conference on Geometric
  Topology (Warsaw, 1978)}, pages 81--102, Warsaw, 1980. PWN.

\bibitem{ek:gen-funct-calc}
Samuel Eilenberg and G.~M. Kelly.
\newblock A generalization of the functorial calculus.
\newblock {\em J. Algebra}, 3:366--375, 1966.

\bibitem{ekmm}
A.~D. Elmendorf, I.~Kriz, M.~A. Mandell, and J.~P. May.
\newblock {\em Rings, Modules, and Algebras in Stable Homotopy Theory},
  volume~47 of {\em Mathematical Surveys and Monographs}.
\newblock American Mathematical Society, 1997.
\newblock With an appendix by M. Cole.

\bibitem{fy:brd-cpt}
Peter~J. Freyd and David~N. Yetter.
\newblock Braided compact closed categories with applications to
  low-dimensional topology.
\newblock {\em Adv. Math.}, 77(2):156--182, 1989.

\bibitem{gn:higher}
Ross Geoghegan and Andrew Nicas.
\newblock Higher {E}uler characteristics. {I}.
\newblock {\em Enseign. Math. (2)}, 41(1-2):3--62, 1995.

\bibitem{GN:trace}
Ross Geoghegan and Andrew Nicas.
\newblock Fixed point theory and the {$K$}-theoretic trace.
\newblock In {\em Nielsen theory and {R}eidemeister torsion ({W}arsaw, 1996)},
  volume~49 of {\em Banach Center Publ.}, pages 137--149. Polish Acad. Sci.,
  Warsaw, 1999.

\bibitem{hasegawa}
Masahito Hasegawa.
\newblock The uniformity principle on traced monoidal categories.
\newblock {\em Publ. Res. Inst. Math. Sci.}, 40(3):991--1014, 2004.

\bibitem{hovey:modelcats}
Mark Hovey.
\newblock {\em Model Categories}, volume~63 of {\em Mathematical Surveys and
  Monographs}.
\newblock American Mathematical Society, 1999.

\bibitem{H:generalized}
S.~Y. Husseini.
\newblock Generalized {L}efschetz numbers.
\newblock {\em Trans. Amer. Math. Soc.}, 272(1):247--274, 1982.

\bibitem{j:nielsen}
Bo~Ju Jiang.
\newblock {\em Lectures on {N}ielsen fixed point theory}, volume~14 of {\em
  Contemporary Mathematics}.
\newblock American Mathematical Society, Providence, R.I., 1983.

\bibitem{ptj:elephant2}
Peter~T. Johnstone.
\newblock {\em Sketches of an Elephant: A Topos Theory Compendium: Volume 2}.
\newblock Number~43 in Oxford Logic Guides. Oxford Science Publications, 2002.

\bibitem{js:geom-tenscalc-i}
Andr{\'e} Joyal and Ross Street.
\newblock The geometry of tensor calculus. {I}.
\newblock {\em Adv. Math.}, 88(1):55--112, 1991.

\bibitem{js:brd-tensor}
Andr{\'e} Joyal and Ross Street.
\newblock Braided tensor categories.
\newblock {\em Adv. Math.}, 102(1):20--78, 1993.

\bibitem{jsv:traced-moncat}
Andr{\'e} Joyal, Ross Street, and Dominic Verity.
\newblock Traced monoidal categories.
\newblock {\em Math. Proc. Cambridge Philos. Soc.}, 119(3):447--468, 1996.

\bibitem{jt:galois}
Andr{\'e} Joyal and Myles Tierney.
\newblock An extension of the {G}alois theory of {G}rothendieck.
\newblock {\em Mem. Amer. Math. Soc.}, 51(309):vii+71, 1984.

\bibitem{kl:cpt}
G.~M. Kelly and M.~L. Laplaza.
\newblock Coherence for compact closed categories.
\newblock {\em J. Pure Appl. Algebra}, 19:193--213, 1980.

\bibitem{jk:frob-2dtqft}
Joachim Kock.
\newblock {\em Frobenius algebras and 2{D} topological quantum field theories},
  volume~59 of {\em London Mathematical Society Student Texts}.
\newblock Cambridge University Press, Cambridge, 2004.

\bibitem{lms:equivariant}
L.~G. Lewis, Jr., J.~P. May, M.~Steinberger, and J.~E. McClure.
\newblock {\em Equivariant stable homotopy theory}, volume 1213 of {\em Lecture
  Notes in Mathematics}.
\newblock Springer-Verlag, Berlin, 1986.
\newblock With contributions by J. E. McClure.

\bibitem{LR:equiv}
Wolfgang L{\"u}ck and Jonathan Rosenberg.
\newblock The equivariant {L}efschetz fixed point theorem for proper cocompact
  {$G$}-manifolds.
\newblock In {\em High-dimensional manifold topology}, pages 322--361. World
  Sci. Publ., River Edge, NJ, 2003.

\bibitem{lurie:tft}
Jacob Lurie.
\newblock On the classification of topological field theories.
\newblock arXiv:0905.0465.

\bibitem{maltsiniotis:traces}
G.~Maltsiniotis.
\newblock Traces dans les cat\'egories mono\"\i dales, dualit\'e et
  cat\'egories mono\"\i dales fibr\'ees.
\newblock {\em Cahiers Topologie G\'eom. Diff\'erentielle Cat\'eg.},
  36(3):195--288, 1995.

\bibitem{mm:eqosp-smod}
M.~A. Mandell and J.~P. May.
\newblock Equivariant orthogonal spectra and {$S$}-modules.
\newblock {\em Mem. Amer. Math. Soc.}, 159(755):x+108, 2002.

\bibitem{mmss}
M.~A. Mandell, J.~P. May, S.~Schwede, and B.~Shipley.
\newblock Model categories of diagram spectra.
\newblock {\em Proc. London Math. Soc. (3)}, 82(2):441--512, 2001.

\bibitem{may:traces}
J.~P. May.
\newblock The additivity of traces in triangulated categories.
\newblock {\em Adv. Math.}, 163(1):34--73, 2001.

\bibitem{maysig:pht}
J.~P. May and J.~Sigurdsson.
\newblock {\em Parametrized homotopy theory}, volume 132 of {\em Mathematical
  Surveys and Monographs}.
\newblock American Mathematical Society, Providence, RI, 2006.

\bibitem{N:trace}
Andrew Nicas.
\newblock Trace and duality in symmetric monoidal categories.
\newblock {\em $K$-Theory}, 35(3-4):273--339 (2006), 2005.

\bibitem{penrose:negdimten}
Roger Penrose.
\newblock Applications of negative dimensional tensors.
\newblock In {\em Combinatorial {M}athematics and its {A}pplications ({P}roc.
  {C}onf., {O}xford, 1969)}, pages 221--244. Academic Press, London, 1971.

\bibitem{equiv}
Kate Ponto.
\newblock Equivariant fixed point theory.
\newblock arXiv:0910.1274.

\bibitem{kate:traces}
Kate Ponto.
\newblock Fixed point theory and trace for bicategories.
\newblock {\em Ast\'erisque}, (333):xii+102, 2010.
\newblock arXiv:0807.1471.

\bibitem{kate:rel}
Kate Ponto.
\newblock Relative fixed point theory.
\newblock {\em Algebr. Geom. Topol.}, 11(2):839--886, 2011.

\bibitem{PS4}
Kate Ponto and Michael Shulman.
\newblock The multiplicativity of fixed point invariants.
\newblock arXiv:1203.0950.

\bibitem{PS2}
Kate Ponto and Michael Shulman.
\newblock Shadows and traces for bicategories.
\newblock {\em Journal of Homotopy and Related Sturctures}.
\newblock 10.1007/s40062-012-0017-0. arXiv:0910.1306.

\bibitem{PS3}
Kate Ponto and Michael Shulman.
\newblock Duality and trace in indexed monoidal categories.
\newblock {\em Theory and Applications of Categories}, 26(23):582--659, 2012.
\newblock arXiv:1211.1555.

\bibitem{rr:shvs-modules}
Pedro Resende and Elias Rodrigues.
\newblock Sheaves as modules.
\newblock {\em Appl. Categ. Structures}, 18(2):199--217, 2010.

\bibitem{selinger:graphical}
Peter Selinger.
\newblock A survey of graphical languages for monoidal categories.
\newblock In Bob Coeke, editor, {\em New Structures for Physics}, chapter~4.
  Springer, 2011.
\newblock Available at
  \url{http://www.mscs.dal.ca/~selinger/papers.html\#graphical} and
  arXiv:0908.3347.

\bibitem{td:groups}
Tammo tom Dieck.
\newblock {\em Transformation groups}, volume~8 of {\em de Gruyter Studies in
  Mathematics}.
\newblock Walter de Gruyter \& Co., Berlin, 1987.

\bibitem{ulrich}
Hanno Ulrich.
\newblock {\em Fixed point theory of parametrized equivariant maps}, volume
  1343 of {\em Lecture Notes in Mathematics}.
\newblock Springer-Verlag, Berlin, 1988.

\bibitem{W:equiv}
Julia Weber.
\newblock The universal functorial equivariant {L}efschetz invariant.
\newblock {\em $K$-Theory}, 36(1-2):169--207 (2006), 2005.

\end{thebibliography}

\end{document}